\documentclass[a4paper,12pt,reqno]{amsart}

\usepackage[T1]{fontenc}
\usepackage[utf8]{inputenc}
\usepackage{lmodern}
\usepackage[english]{babel}
\usepackage[%
	left=2.5cm,       
	right=2.5cm,      
	top=3.5cm,        
	bottom=3.5cm,     
	heightrounded,    
	bindingoffset=0mm 
]{geometry}

\usepackage{amsmath}
\usepackage{amssymb}
\usepackage{mathtools}
\usepackage{mathrsfs}
\usepackage[normalem]{ulem}

\usepackage{microtype}
\usepackage[
colorlinks=true,
urlcolor=teal,
citecolor=magenta,
linkcolor=teal,
breaklinks=true]
{hyperref}

\usepackage[noabbrev,capitalize,nameinlink]{cleveref}

\usepackage{bookmark}

\usepackage[initials,msc-links,
]{amsrefs}

\usepackage[dvipsnames]{xcolor}
\usepackage{graphicx}
\usepackage{url}

\usepackage{enumitem}

\numberwithin{equation}{section}

\theoremstyle{plain}
\newtheorem{theorem}{Theorem}[section]
\newtheorem{lemma}[theorem]{Lemma}
\newtheorem{proposition}[theorem]{Proposition}

\theoremstyle{definition}
\newtheorem{definition}[theorem]{Definition}

\newtheorem{remark}[theorem]{Remark}

\newcommand{\content}{{\mathbb H}}
\newcommand{\di}{\,\mathrm{d}}
\newcommand{\heat}{\mathsf{h}}
\newcommand{\Heat}{\mathsf{H}}
\newcommand{\lap}{L}
\newcommand{\N}{\mathbb{N}}

\newcommand{\R}{\mathbb{R}}
\newcommand{\sob}[1]{\mathcal{S}^{#1}}

\newcommand{\sus}{\lambda}
\newcommand{\weakstarto}{\overset{\star}{\rightharpoonup}}
\newcommand{\e}{\varepsilon}

\newcommand{\mres}{\mathbin{\vrule height 1.6ex depth 0pt width
0.13ex\vrule height 0.13ex depth 0pt width 1.3ex}}

\DeclarePairedDelimiter{\set}{\{}{\}}

\usepackage[
final
]{showlabels}

\showlabels{bib}

\allowdisplaybreaks
\begin{document}

\title[Sharp BBM formula and asymptotics of heat content-type energies]{Sharp conditions for the BBM formula \\ and asymptotics of heat content-type energies}

\author[L.~Gennaioli]{Luca Gennaioli}
\address[L.~Gennaioli]{University of Warwick, Mathematics Research Centre, Coventry CV4 7AL, UK}
\email{luca.gennaioli@warwick.ac.uk}

\author[G.~Stefani]{Giorgio Stefani}
\address[G.~Stefani]{Università degli Studi di Padova, Dipartimento di Matematica ``Tullio Levi-Civita'', Via Trieste 63, 35121 Padova (PD), Italy}
\email{giorgio.stefani@unipd.it}

\date{\today}



\thanks{\textit{Acknowledgements}. 
The first-named author is supported by UK Research and Innovation (UKRI) under the Horizon Europe funding guarantee (grant agreement No.\ EP/Z000297/1).
The second-named author is member of the Istituto Nazionale di Alta Matematica (INdAM), Gruppo Nazionale per
l’Analisi Matematica, la Probabilità e le loro Applicazioni (GNAMPA), and has received funding from INdAM under the INdAM--GNAMPA Project 2025 \textit{Metodi variazionali per problemi dipendenti da operatori frazionari isotropi e anisotropi} (grant agreement No.\ CUP\_E53\-240\-019\-500\-01), and from the European Union -- NextGenerationEU and the University of Padua under the 2023 STARS@UNIPD  Starting Grant Project \textit{New Directions in Fractional Calculus -- NewFrac} (grant agreement No.\ CUP\_C95\-F21\-009\-990\-001).
The present research started while the authors were doctoral student and postdoctoral researcher, respectively, at the Scuola Internazionale Superiore
di Studi Avanzati (SISSA) in Trieste, Italy. The authors would like to express their gratitude to the institution for providing excellent working conditions and a stimulating atmosphere. 
The authors would also like to thank Luca Rizzi for his insightful questions regarding the fractional heat content from which the present work originated. 
}

\begin{abstract}
Given $p\in[1,\infty)$, we provide sufficient and necessary conditions on the non-negative measurable kernels $(\rho_t)_{t\in(0,1)}$ ensuring  convergence of the associated Bour\-gain--Bre\-zis--Mironescu (BBM) energies $(\mathscr F_{t,p})_{t\in(0,1)}$ to a variant of the $p$-Dirichlet energy on $\R^N$ as $t\to0^+$ both in the pointwise and in the $\Gamma$-sense. 
We also devise sufficient conditions on $(\rho_t)_{t\in(0,1)}$ yielding local compactness in $L^p(\R^N)$ of sequences with bounded BBM energy.
Moreover, we give sufficient conditions on $(\rho_t)_{t\in(0,1)}$ implying pointwise and $\Gamma$-convergence and equicoercivity of $(\mathscr F_{t,p})_{t\in(0,1)}$ when the limit $p$-energy is of non-local type.
Finally, we apply our results to provide asymptotic formulas in the pointwise and $\Gamma$-sense for heat content-type energies both in the local and non-local settings. 
\end{abstract}

\maketitle

\section{Introduction}

\subsection{Framework}

We let $I=(0,1)$ and we fix a family $(\rho_t)_{t\in I}\subset L^1_{\rm loc}(\R^N)$ of non-negative functions.
Given $p\in[1,\infty)$ and $t\in I$, we consider the non-local functionals $\mathscr F_{t,p}\colon L^p(\R^N)\to[0,\infty]$ defined as
\begin{equation}
\label{eqi:F_tp}
\begin{split}
\mathscr F_{t,p}(u)
%&=
%\int_{\R^N}\frac{\|u(\cdot+z)-u\|_{L^p}^p}{|z|^p}\,\rho_t(z)\di z
%\\
%&
=
\int_{\R^N}
\int_{\R^N}
\frac{|u(x)-u(y)|^p}{|x-y|^p}\,\rho_t(x-y)\di x\di y
\end{split}
\end{equation} 
for all $u\in L^p(\R^N)$.

The asymptotic behavior of the family $(\mathscr{F}_{t,p})_{t \in I}$ as $t \to 0^+$ was first investigated by Bourgain, Brezis and Mironescu in their seminal paper~\cite{BBM01}.
Their work has inspired a vast literature on non-local-to-local convergence results (or \emph{BBM formulas}, after~\cite{BBM01}). 

While a comprehensive overview of the research is beyond our scope, we focus on the results most aligned with the spirit of~\cite{BBM01}. 
In particular, we refer to~\cites{D02,P04a,P04b,MS02,MS03,ArB23} for foundational contributions and to~\cites{LS11,LS14} for extensions to general open sets. 
For $\Gamma$-convergence results, see~\cites{ADM11,BP19,DKP22,CDKNP23,KL25,P20}. 
For energies derived from gradient-type integro-differential operators, we refer to~\cites{M12,SM19,BCCS22,CS23,BCM21,MD16,BM23 }. 
For the extension to non-Euclidean frameworks, such as magnetic Sobolev spaces, Riemannian manifolds, Carnot groups, and metric-measure spaces, see~\cites{Bar11,KM19,DS19,LPZ24a,LPZ24b,H24,G22,HP21,GT23,GT24,NPSV20}. 
For other strictly related results, we also refer to the monographs~\cites{MRT19,AABPT23}.

BBM formulas play a central role in several applications in modern Analysis.
Far from being complete, we refer to Brezis’ celebrated work~\cites{B02} on how to recognize constant functions, to~\cites{B15} for applications to image denoising,
to~\cite{DDG24} for extensions accounting for antisymmetric exchange interactions, 
and to \cites{ABSS25,S24} for the study of the convergence of non-local Ginzburg--Landau functionals.

\subsection{Sharp conditions}

A common trait of the works mentioned above is that they only concern \textit{sufficient} conditions
for BBM formulas to hold.
Namely, in the specific case of the functionals in~\eqref{eqi:F_tp}, under a certain set of conditions on the family $(\rho_t)_{t\in I}$, one can find an infinitesimal sequence $(t_k)_{k\in\N}\subset I$ and a non-negative Radon measure $\mu\in\mathscr M^+(\mathbb S^{N-1})$ on the $(N-1)$-dimensional sphere~$\mathbb S^{N-1}$ in $\R^N$, depending on $(\rho_{t_k})_{k\in\N}$ only, such that 
\begin{equation}
\label{eqi:BBM}
\lim_{k\to\infty}
\mathscr F_{t_k,p}(u)
=
\mathscr D_p^\mu(u)
\end{equation}
for every $u\in {\sob{p}}(\R^N)$,
where 
\begin{equation}
\label{eqi:p-Dir}
\mathscr D_p^\mu(u)
=
\int_{\mathbb S^{N-1}}\|\sigma\cdot Du\|_{L^p}^p\di\mu(\sigma).
\end{equation}
Here and in the following, we let 
\begin{equation*}
{\sob{p}}(\R^N)
=
\begin{cases}
W^{1,p}(\R^N)
&
\text{for}\ p>1,
\\[1ex]
BV(\R^N)
&
\text{for}\ p=1,
\end{cases}
\end{equation*}
and we let $Du$ be the distributional gradient of $u\in {\sob{p}}(\R^N)$ (if $p=1$, then $Du$ may be a finite Radon measure on $\R^N$).
Moreover, for every $\sigma\in\mathbb S^{N-1}$ and $u\in {\sob{p}}(\R^N)$, we let
\begin{equation*}
\|\sigma\cdot Du\|_{L^p}^p
=
\begin{cases}
\displaystyle
\int_{\R^N}|\sigma\cdot Du(x)|^p\di x
&
\text{for}\
p>1,
\\[3ex]
|\sigma\cdot Du|(\R^N)
&
\text{for}\
p=1.
\end{cases}
\end{equation*}

In the recent paper~\cite{DDP24}, the authors devise a set of conditions on the family $(\rho_t)_{t\in I}$ that are both \textit{sufficient} and \textit{necessary} for the validity of~\eqref{eqi:BBM} in the case $p=2$ by means of Fourier transform techniques.
Precisely, they recast the functionals $(\mathscr F_{t,2})_{t\in I}$ in~\eqref{eqi:F_tp} into double integrals of the form
\begin{equation*}
v\mapsto
\int_{\R^N}|v(\xi)|^2\int_{\R^N}\frac{1-\cos(z\cdot\xi)}{|z|^2}\,\rho_t(z)\di z\di\xi
\end{equation*}  
where $v\in L^2(\R^N)$ is such that $|\cdot|\,|v|\in L^2(\R^N)$. 
Unfortunately, for $p\ne2$, the Fourier approach is not viable anymore, but in~\cite{DDP24}*{Sec.~5.3} the authors conjecture that similar conditions are sufficient and necessary for the validity of~\eqref{eqi:BBM} for every $p\in[1,\infty)$.

For \textit{radially symmetric} families $(\rho_t)_{t\in I}$, necessary conditions are outlined in~\cite{FK24} for $p=2$, while sufficient and necessary conditions are achieved in~\cite{F25} for every $p>1$, confirming the conjecture made in~\cite{DDP24} in this particular case.     

To the best of our knowledge, the conjecture in~\cite{DDP24} is currently open for arbitrary families $(\rho_t)_{t\in I}$ and $p\ne2$.
Our first main result, stated in \cref{resi:bbm} below,  affirmatively answers the conjecture posed in~\cite{DDP24}.
Even more, we prove that the conditions devised in~\cite{DDP24} are sufficient and necessary for the convergence of the functionals in~\eqref{eqi:F_tp} not only in the \textit{pointwise} sense, but also in the \emph{$\Gamma$-convergence} sense (for a complete description of \emph{$\Gamma$-convergence}, we refer to the monographs~\cites{Braides02,DalMaso93}).
For the notation, we refer to \cref{sec:preliminaries}.

\begin{theorem}
\label{resi:bbm}

Let $p\in[1,\infty)$. The following are equivalent.

\begin{enumerate}[label=(\Alph*),itemsep=1ex,topsep=1ex,leftmargin=5ex]

\item 
\label{itemi:bbm_kernels}
There exists an infinitesimal sequence $(t_k)_{k\in\N}\subset I$ such that 
\begin{equation}
\label{eqi:bbm_suff}
\sup_{R>0}
\limsup_{k\to\infty}
R^p\int_{\R^N}\frac{\rho_{t_k}(z)}{R^p+|z|^p}\di z<\infty
\end{equation}
and $\nu_k=\rho_{t_k}\mathscr L^N\weakstarto\alpha\delta_0$ in $\mathscr M_{\rm loc}(\R^N)$ as $k\to\infty$ for some $\alpha\ge0$.

\item
\label{itemi:bbm_limits}
There exist an infinitesimal sequence $(t_k)_{k\in\N}\subset I$ and $\mu\in\mathscr M^+(\mathbb S^{N-1})$
such that:
\begin{enumerate}[label=(\roman*),itemsep=1ex,topsep=2ex,leftmargin=1.5em]
\item
if $u\in {\sob{p}}(\R^N)$, then
$\displaystyle
\limsup_{k\to\infty}
\mathscr F_{t_k,p}(u)
\le
\mathscr D_p^\mu(u)$;

\item
\label{iitemi:bbm_liminf}
if $(u_k)_{k\in\N}\subset L^p(\R^N)$ is such that $u_k\to u$ in $L^p(\R^N)$ as $k\to\infty$ for some $u\in {\sob{p}}(\R^N)$, then
$\displaystyle
\liminf_{k\to\infty}
\mathscr F_{t_k,p}(u_k)
\ge 
\mathscr D_p^\mu(u)$.

\end{enumerate}
 
\end{enumerate}
\end{theorem}

In parts \ref{itemi:bbm_kernels} and \ref{itemi:bbm_limits} of \cref{resi:bbm}, the sequence $(t_k)_{k \in\N}$ is not necessarily the same. We refer to \cref{rem:tk_uguale} for an example where one must pass to a subsequence of $(t_k)_{k\in\N}$ in the implication \ref{itemi:bbm_kernels}$\implies$\ref{itemi:bbm_limits}.  
Moreover, as observed in~\cites{P04a,LS11}, property~\ref{iitemi:bbm_liminf} in~\ref{itemi:bbm_limits} can be further refined by additionally requiring that the family $(\rho_t)_{t\in I}$ has \emph{maximal rank} (see \cref{def:maxrank} below for the precise formulation).

\begin{theorem}\label{resi:maxrank}
Let $p\in[1,\infty)$. Assume that~\ref{itemi:bbm_kernels} or~\ref{itemi:bbm_limits} holds and that $(\rho_{t_k})_{k\in\N}$ has maximal rank. 
If $(u_k)_{k\in\N}\subset L^p(\R^N)$ is such that $u_k\to u$ in $L^p(\R^N)$ as $k\to\infty$ for some $u\in L^p(\R^N)$ and 
$\displaystyle
\liminf_{k\to\infty} \mathscr F_{t_k,p}(u_k)<\infty$,
then $u\in {\sob{p}}(\R^N)$.
\end{theorem}
  
When the conclusion of \cref{resi:maxrank} holds---that is, when the finiteness of the $\liminf$ of the functionals along an $L^p$ convergent sequence implies that the limit function belongs to some subspace $\mathcal X^p(\R^N)$ of $L^p(\R^N)$ (e.g., $\mathcal X^p(\R^N)=\sob{p}(\R^N)$)---we say that  the functionals $(\mathscr F_{t,p})_{t\in I}$ are \emph{coercive} on $\mathcal X^p(\R^N)$ (see \cref{def:coercive} below for a more precise statement).

Any radially symmetric family $(\rho_t)_{t\in I}$ has maximal rank, but non-radially symmetric families with maximal rank are known (examples can be found in~\cite{P04a}).
We do not know if the maximal rank condition is also necessary for the conclusion of \cref{resi:maxrank} to hold.

To prove \cref{resi:bbm,resi:maxrank}, we mix the approaches of~\cites{DDP24,LS11,P04a} in a new fashion.

As in the proof of~\cite{DDP24}*{Th.~1.2}, the  sufficiency part consists in showing that~\eqref{eqi:bbm_suff} yields the pointwise and $\Gamma$-convergence of $(\mathscr F_{t,p})_{t\in I}$ as $t\to0^+$ (up to subsequences) to
\begin{equation}
\label{eqi:G_munu}
\mathscr G^{\mu,\nu}_p(u)
=
\int_{\mathbb S^{N-1}}
\|\sigma\cdot Du\|_{L^p}^p\di\mu(\sigma)
+
\int_{\R^N\setminus\set{0}}
\frac{\|u(\cdot+z)-u\|_{L^p}^p}{|z|^p}\di \nu(z)
\end{equation}
for all $u\in {\sob{p}}(\R^N)$, where $\mu\in\mathscr M^+(\mathbb S^{N-1})$ and $\nu\in\mathscr M^+(\R^N)$ are two non-negative finite Radon measures  depending on $(\rho_t)_{t\in I}$ only (see \cref{res:munu_p} for the precise statement).
Loosely speaking, the measure $\mu$ is given by (up to subsequences) 
\begin{equation}
\label{eqi:mu_loose}
\mu(E)
=
\lim_{\delta\to0^+}
\lim_{t\to0^+}
\int_E\left(
\int_0^\delta\rho_t(\sigma r)\,r^{N-1}\di r
\right)\di\mathscr H^{N-1}(\sigma)
\end{equation}
for every Borel set $E\subset\mathbb S^{N-1}$, while the measure $\nu$ is (up to subsequences) the weak$^\star$ limit of the family $(\rho_t\mathscr L^N)_{t\in I}$ as $t\to0^+$.
Consequently, due to~\eqref{eqi:G_munu}, in order to achieve~\eqref{eqi:BBM}, the measure $\nu$ must be supported on~$\set*{0}$; that is, $\nu=\alpha\delta_0$ for some $\alpha\ge0$.
Additionally, thanks to~\eqref{eqi:mu_loose}, as in~\cites{P04a,LS11} the maximal rank assumption guarantees that the limit  $p$-Dirichlet energy~\eqref{eqi:p-Dir} bounds the ${\sob{p}}(\R^N)$ seminorm, yielding \cref{resi:maxrank}.
To prove the convergence to $\mathscr G_p^{\mu,\nu}$, we revise the line of~\cite{DDP24} replacing Fourier transform techniques with some plain arguments invoking basic properties of ${\sob{p}}(\R^N)$ functions.

The proof of the necessary part differs from the one of~\cite{DDP24} and combines three ingredients.
We first show that the validity of~\eqref{eqi:BBM} for some $\mu\in\mathscr M^+(\mathbb S^{N-1})$ implies~\eqref{eqi:bbm_suff} by testing~\eqref{eqi:BBM} on suitably chosen compactly supported Lipschitz functions.
This, in turn, implies that $(\mathscr F_{t,p})_{t\in I}$ converges to $\mathscr G_p^{\widetilde\mu,\nu}$ as $t\to0^+$ for some $\widetilde\mu\in\mathscr M^+(\mathbb S^{N-1})$ and $\nu\in\mathscr M^+(\R^N)$ as above. 
We hence conclude the proof by showing that, if $\mathscr G_p^{\mu,0}(u)=\mathscr G_p^{\widetilde\mu,\nu}(u)$ for all $u\in {\sob{p}}(\R^N)$, then  $\nu=\alpha\delta_0$ for some $\alpha\ge0$ via a scaling argument.

\subsection{Compactness}

A further research line concerns equicoercivity properties of the functionals~\eqref{eqi:F_tp}. 
As well-known, for $N\ge2$, the radial symmetry of the family $(\rho_t)_{t\in I}$ yields equicoercivity of the functionals~\eqref{eqi:F_tp} (that is, compactness of their sublevel sets), see~\cite{BBM01}*{Th.~4}, \cite{P04b}*{Ths.~1.2 and~1.3} and~\cite{AABPT23}*{Th.~4.2}. 
If $N=1$, then additional conditions must be imposed due to counterexamples~\cites{P04b,BBM01}.

To the best of our knowledge, no equicoercivity result is available for non-radially symmetric families.
Our second main result, inspired by~\cite{BP19}*{Th.~3.5}, partially fills this gap and yields quite flexible  \textit{sufficient} conditions on the possibly non-radially symmetric family $(\rho_t)_{t\in I}$ which ensure the equicoercivity of the functionals in~\eqref{eqi:F_tp}.

In more precise terms, given $p\in[1,\infty)$, we consider
\begin{equation}
\label{eqi:rho_special}
\rho_t(x)
=
\frac{|x|^p K_t(x)}{\phi_{K,\beta,p}(t)},
\quad
\text{for}\ x\in\R^N\ \text{and}\ t\in I, 
\end{equation}
where $(K_t)_{t\in\N}$ is given by
\begin{equation*}
K_t(x)
=
\beta(t)^N
K(\beta(t)x),
\quad
\text{for a.e.}\ x\in\R^N\ \text{and}\ t\in I, 
\end{equation*}
for some non-negative function $K\not\equiv0$ such that $|\cdot|^p\,K\in L^1_{\rm loc}(\R^N)$ and a Borel function $\beta\colon I\to(0,\infty)$.
Moreover, in~\eqref{eqi:rho_special}, we have set
\begin{equation*}
\phi_{K,\beta,p}(t)
=\frac{m_{K,p}(\beta(t))}{\beta(t)^p}
\quad
\text{for all}\ t\in I, 
\end{equation*}
where $m_{K,p}\colon[0,\infty)\to[0,\infty)$ is defined as
\begin{equation*}
m_{K,p}(R)=\int_{B_R}|x|^p\,K(x)\di x
\quad
\text{for all}\ R>0.
\end{equation*}
We refer to \cref{subsec:special_kernels} for a more detailed description of the family in~\eqref{eqi:rho_special}.

With the above notation in force, our result can be stated as follows (see \cref{def:locprecomp} for the notion of  \emph{local precompactness}).

\begin{theorem}%[Compactness]
\label{resi:compactness}
With the above notation in force, assume that 
\begin{equation*}
|\cdot|^p\,K\in L^1(\R^N)
\quad
\text{and}
\quad
\lim_{t\to0^+}\beta(t)=\infty.
\end{equation*}
If $(t_k)_{k\in\N}\subset I$ is infinitesimal and $(u_k)_{k\in\N}\subset L^p(\R^N)$ is such that
\begin{equation*}
\sup_{k\in\N}
\big(
\|u_k\|_{L^p}
+
\mathscr F_{t_k,p}(u_k)
\big)<\infty,
\end{equation*}
then $(u_k)_{k\in\N}$ is locally precompact in $L^p(\R^N)$ and any of its $L^p_{\rm loc}(\R^N)$ limits is in ${\sob{p}}(\R^N)$.
\end{theorem}

The proof of \cref{resi:compactness} generalizes the strategy in~\cite{BP19} to any $p\in[1,\infty)$.
The core idea is to show that, for each $u\in L^p(\R^N)$ and $t>0$, there exists $v_t\in {\sob{p}}(\R^N)$ such that 
\begin{equation*}
\|v_t-u\|_{L^p}
\le 
C_{K,p}
\,
\mathscr F^K_{t,p}(u)
\,
\beta(t)^{-p}
\quad
\text{and}
\quad
\|\nabla v_t\|_{L^p}
\le
C_{K,p} 
\,\mathscr F^K_{t,p}(u),
\end{equation*}
where $C_{K,p}>0$ depends on $K$ and $p$ only (see \cref{res:supcomp_bound} for the precise statement).

\subsection{Non-local limit energies}

The convergence of $(\mathscr F_{t,p})_{t \in I}$ to the functional $\mathscr G_p^{\mu,\nu}$ in~\eqref{eqi:G_munu} as $t \to 0^+$ can be seen as a non-local-to-non-local convergence result. Naturally, one may ask whether the stability of the non-local nature of the functionals also occurs for functions with lower regularity. A similar behavior was, in fact, observed in~\cite{AV16}*{Th.~1.1(iii)} for characteristic functions of bounded sets with finite (local or non-local) perimeter.

Our next main result aims to provide a deeper understanding of this non-local stability, thereby generalizing~\cite{AV16}.
Here and below, given $p\in[1,\infty)$ and a measurable function $\kappa\colon\R^N\to[0,\infty]$, we consider the non-local Sobolev space 
\begin{equation*}
W^{\kappa,p}(\R^N)  
=  
\set*{u\in L^p(\R^N) : [u]_{W^{\kappa,p}}<\infty},  
\end{equation*}  
where the non-local seminorm is defined by letting  
\begin{equation*}  
[u]_{W^{\kappa,p}}  
=  
\left(\int_{\R^N}\int_{\R^N}|u(x)-u(y)|^p\,\kappa(x-y)\di x \di y\right)^{1/p}.  
\end{equation*}  
We refer, e.g., to~\cites{BS24,F25} for more details on the space $W^{\kappa,p}(\R^N)$. 
Here we just observe that the \emph{fractional Sobolev--Slobodeckij space} $W^{s,p}(\R^N)$, with $s\in(0,1)$ and $p\in[1,\infty)$, corresponds to the choice $\kappa(z)=|z|^{-N-sp}$ for all $z\in\R^N\setminus\set*{0}$. 

\begin{theorem}%[Non-local stability]
\label{resi:diretto_lim}
Let $p\in[1,\infty)$ and $(\rho_t)_{t\in I}\subset L^1_{\rm loc}(\R^N)$.
Assume that there exist $C>0$ and a measurable function $\kappa\colon\R^N\to[0,\infty]$ such that
\begin{equation}
\label{eqi:diretto_c}
\frac{\rho_t(z)}{|z|^p}
\le 
C\kappa(z)
\quad
\text{for all}\ t\in I\ \text{and a.e.}\ z\in\R^N
\end{equation}
and
\begin{equation}
\label{eqi:diretto_lim}
\lim_{t\to0^+}\frac{\rho_t(z)}{|z|^p}=\kappa(z)
\quad
\text{for a.e.}\ z\in\R^N.
\end{equation}
Then, the limit
\begin{equation*}
\lim_{t\to0^+}
\mathscr F_{t,p}(u)
=
[u]_{W^{\kappa,p}}^p,
\quad
\text{for}\ u\in W^{\kappa,p}(\R^N),
\end{equation*}
holds in the pointwise sense and in the $\Gamma$-sense with respect to the $L^p$ topology, and the functionals $(\mathscr F_{t,p})_{t\in I}$ are coercive on $W^{\kappa,p}(\R^N)$.
\end{theorem}  

As a natural analogue of \cref{resi:compactness} in this setting, we now complement \cref{resi:diretto_lim} with the following compactness result.

\begin{theorem}%[Compactness, non-local version]
\label{resi:diretto_comp}
Let $p\in[1,\infty)$ and $(\rho_t)_{t\in I}\subset L^1_{\rm loc}(\R^N)$.
Assume that, for every $\e>0$, there exists $\delta\in(0,1]$ such that 
\begin{equation}
\label{eqi:diretto_eps-delta}
\frac{\rho_t(z)}{|z|^p}
\ge
\frac{1}{\e\delta^N}
\quad
\text{for a.e.}\ z\in B_\delta\ \text{and every}\ t\in(0,\delta). 
\end{equation}
If $(t_k)_{k\in\N}\subset I$ is infinitesimal and $(u_k)_{k\in\N}\subset L^p(\R^N)$ is such that
\begin{equation*}
\sup_{k\in\N}
\big(
\|u_k\|_{L^p}
+
\mathscr F_{t_k,p}(u_k)
\big)<\infty,
\end{equation*}
then $(u_k)_{k\in\N}$ is locally precompact in $L^p(\R^N)$ and any of its $L^p_{\rm loc}(\R^N)$ limits is in $W^{\kappa,p}(\R^N)$,
where $\kappa\colon\R^N\to[0,\infty]$ is given by 
\begin{equation*} \kappa(z)=\liminf_{t\to0^+}\frac{\rho_t(z)}{|z|^p}
\quad
\text{for a.e.\ $z\in\R^N$.}
\end{equation*}
\end{theorem}

\cref{resi:diretto_lim} relies on an application of the Dominated Convergence Theorem and Fatou's Lemma, exploiting the limits~\eqref{eqi:diretto_c} and~\eqref{eqi:diretto_lim}.
\cref{resi:diretto_comp}, instead, is a consequence of the Fréchet--Kolmogorov Compactness Theorem, since the bound~\eqref{eqi:diretto_eps-delta} allows to quantitatively control the $L^p$ distance between functions and their smoothed versions.

The assumptions of \cref{resi:diretto_lim,resi:diretto_comp} are naturally motivated by the application of these results to families induced by \emph{fractional} and \emph{nonlocal heat kernels}; see \cref{resi:frac_heat} (in the case $2s < p$) and \cref{res:acuna} (in the case $\alpha \in \mathscr C_p$) below.  
Although these assumptions are broad enough to cover all the applications considered in this paper, we do not claim that they are sharp, and we leave the analysis of sharp conditions for future work.  
Nevertheless, we refer to~\cite{S25} for generalizations of these results.

\subsection{Asymptotics of heat-type energies}

We apply our results to study the asymptotic behavior of energies induced by heat-type kernels.

Our first main result in this direction concerns the classical \emph{heat semigroup}
$(\Heat_t)_{t>0}$, see \cref{subsec:heat} for the precise definition.

\begin{theorem}
\label{resi:heat}
If $p\in[1,\infty)$, then the limit
\begin{equation}
\label{eqi:heat}
\lim_{t\to0^+}
{t^{-\frac p2}}
\int_{\R^N}\Heat_t(|u-u(x)|^p)(x)\di x
=
\frac{2\Gamma(p)}{\Gamma(p/2)}
\,
\|Du\|_{L^p}
^p,
\quad
\text{for}\ u\in {\sob{p}}(\R^N),
\end{equation}
holds in the pointwise and $\Gamma$-sense with respect to the $L^p$ topology, and the functionals on the left-hand side are coercive on ${\sob{p}}(\R^N)$. 
Moreover, if $(t_k)_{k\in\N}\subset I$ is infinitesimal  and $(u_k)_{k\in\N}\subset L^p(\R^N)$ is such that
\begin{equation*}
\liminf_{k\to\infty}
{t^{-\frac p2}_k}
\int_{\R^N}\Heat_{t_k}(|u_k-u_k(x)|^p)(x)\di x<\infty,
\end{equation*}
then $(u_k)_{k\in\N}$ is locally precompact in $L^p(\R^N)$ and any of its $L^p_{\rm loc}(\R^N)$ limits is in ${\sob{p}}(\R^N)$.
\end{theorem}

The pointwise limit in \cref{resi:heat}  and its link with the BBM formula are already known, see~\cite{GT24}*{Th.~B} and the related discussion for example.
However, we were not able to trace the $\Gamma$-convergence and compactness parts of \cref{resi:heat} in the literature. 

It is worth mentioning that formula~\eqref{eqi:heat} plays a central role in the study of small time asymptotics for the heat semigroup.
For $p=1$ and $u=\chi_E\in BV(\R^N)$, the limit in~\eqref{eqi:heat} rewrites as
\begin{equation}
\label{eqi:contentE}
Q_E(t)\coloneqq\int_E\Heat_t\chi_E\di x
=
|E|
-
\frac1{\sqrt\pi}\,P(E)\cdot\sqrt t
+
o(\sqrt{t})
\quad
\text{as}\ t\to0^+,
\end{equation}
the so-called (\emph{relative}) \emph{heat content} of the set~$E$.
From a physical perspective, if $E$ is a container that is perfectly insulated at time $t=0$, then $Q_E(t)$ measures the amount of heat that remains inside $E$ at time~$t>0$.

The study of the short-time behavior of heat-semigroup energies originated from De Giorgi's seminal work~\cite{D59}, and later expanded across various settings, including not only Euclidean spaces~\cites{MPPP07,vdBLG94,AMM13,P03,L94}, but also Riemannian manifolds~\cite{AB23}, Carnot groups~\cite{GT24}, sub-Riemannian manifolds~\cites{ARR17,CHT21}, and $\mathsf{RCD}$ metric-measure spaces~\cite{BPP25}.
For smooth sets $E$, the asymptotic expansion in~\eqref{eqi:contentE} continues in powers of $\sqrt t$, with coefficients encoding fundamental geometric features of~$E$, such as mean curvature.

Our second main result is the fractional counterpart of \cref{resi:heat}, dealing with the \emph{fractional heat semigroup} $(\Heat_t^s)_{t>0}$, with $s\in(0,1)$, see \cref{sec:frac_heat_kernel} for the precise definition.

\begin{theorem}
\label{resi:frac_heat}
Given $p\in[1,\infty)$ and $s\in(0,1)$, let $\psi_{s,p}\colon I\to[0,\infty)$ be defined as
\begin{equation*}
\psi_{s,p}(t)
=
\begin{cases}
t^{\frac p{2s}} & \text{if}\ 2s>p,
\\%[1ex]
t|\log t| & \text{if}\ 2s=p,
\\%[1ex]
t & \text{if}\ 2s<p,
\end{cases}
\end{equation*}
for all $t\in I$.
The limits
\begin{equation*}
\lim_{t\to0^+}
\int_{\R^N}\frac{\Heat_t^s(|u-u(x)|^p)(x)}{\psi_{s,p}(t)}\di x
=
\begin{cases}
\displaystyle
\frac{\Gamma\left(1-\frac p{2s}\right)}{\Gamma\left(1-\frac p2\right)}
\,\frac{2\Gamma(p)}{\Gamma(p/2)}
\,
\|Du\|_{L^p}^p
& 
\text{in}\ 
{\sob{p}}(\R^N)\ \text{if}\ 2s\ge p,
\\[5ex]
\displaystyle
\frac{s\,4^s}{\pi^{\frac N2}}
\,
\frac{\Gamma\left(\frac{N}{2}+s\right)}{\Gamma(1-s)}
\,
[u]_{W^{2s,\frac p{2s}}}^p
&  
\text{in}\
W^{2s,\frac p{2s}}(\R^N)\ \text{if}\ 2s<p,
\end{cases}
\end{equation*}
hold in the pointwise and $\Gamma$-sense with respect to the $L^p$ topology, and all the functionals on the left-hand sides are coercive on the respective spaces.
Moreover, if $(t_k)_{k\in\N}\subset I$ is infinitesimal and $(u_k)_{k\in\N}\subset L^p(\R^N)$ is such that
\begin{equation*}
\liminf_{k\to\infty}
\int_{\R^N}\frac{\Heat_{t_k}^s(|u_k-u_k(x)|^p)(x)}{\psi_{s,p}(t_k)}\di x<\infty,
\end{equation*}
then $(u_k)_{k\in\N}$ is locally precompact in $L^p(\R^N)$ and any of its $L^p_{\rm loc}(\R^N)$ limits is in ${\sob{p}}(\R^N)$ if $2s\ge p$ and in $W^{2s,\frac p{2s}}(\R^N)$ if $2s<p$.
\end{theorem}

For $p=1$ and for characteristic functions of sets only, the pointwise limits in \cref{resi:frac_heat} were obtained in~\cite{AV16}*{Th.~4.1} under the additional assumption that the sets under consideration are bounded, while the $\Gamma$-limits were established in the recent work~\cite{KL25}.
The strategies of proof in~\cites{AV16,KL25} are both different from our approach.
We refer to \cref{rem:KL25_const,rem:comparison_acuna} below for further comments on the relation between our work and~\cites{AV16,KL25}.

In fact, we can prove a more general version of \cref{resi:frac_heat},  generalizing~\cite{AV16}*{Th.~1.1}.
For the precise statement, which is more involved, we refer to \cref{res:acuna} below.

\subsection{Other approaches in Hilbert spaces}

In the Hilbertian case $p=2$, \cref{resi:heat,resi:frac_heat} can be achieved in alternative and more general ways.

Let $\mathcal H$ be a Hilbert space, $(\Heat_t)_{t\ge0}$ be a \emph{strongly continuous semigroup of symmetric operators} on $\mathcal H$, and $\lap$ be the (\emph{infinitesimal}) \emph{generator} of $(\Heat_t)_{t\ge0}$ with domain $\mathcal D(\lap)\subset\mathcal H$.
In this setting, the \emph{semigroup content} of $u\in\mathcal H$ is defined as the map 
\begin{equation*}
[0,\infty)\ni t\mapsto\content_t(u)\coloneqq(\Heat_tu,u)_{\mathcal H}.    
\end{equation*}

With this notation in force, we can state our last main result.

\begin{theorem}
\label{resi:hilbert}
Let $\mathcal H$, $(\Heat_t)_{t\ge0}$ and $\lap$ be as above.
The following hold:

\begin{enumerate}[label=(\roman*),itemsep=1ex,topsep=1ex]

\item
\label{itemi:hilbert_limsup}
if $u\in\mathcal D(\lap)$, then 
$\displaystyle
\limsup_{t\to0^+}
\frac{\content_0(u)-\content_t(u)}{t}
\le
(-\lap u,u)_{\mathcal H}$;

\item
\label{itemi:hilbert_liminf}
if $(t_k)_{k\in\N}\subset(0,\infty)$ is infinitesimal and $(u_k)_{k\in\N}\subset\mathcal H$ is such that $u_k\to u$ in $\mathcal H$ as $k\to\infty$ for some $u\in\mathcal D(\lap)$, then 
\begin{equation*}
\liminf_{k\to\infty}
\frac{\content_0(u_k)-\content_{t_k}(u_k)}{t_k}
\ge 
(-\lap u,u)_{\mathcal H}.
\end{equation*}

\end{enumerate}
As a consequence, the functionals $u\mapsto \frac{\content_0(u)-\content_{t}(u)}{t}$ converge to $u\mapsto(-\lap u,u)_{\mathcal H}$ on $\mathcal H$ as $t\to0^+$ pointwise and in the $\Gamma$-sense with respect to the strong topology in $\mathcal H$.
\end{theorem}

The proof of \cref{resi:heat} exploits some elementary arguments involving the spectral representation of the non-negative operator~$-\lap$.
We observe that, in the case $\mathcal H=L^2(\R^N)$ and $\lap=-(-\Delta)^s$ with $s\in(0,1]$, \cref{resi:hilbert} covers the case $p=2$ in \cref{resi:heat,resi:frac_heat}.
Actually, if $\mathcal H=L^2(\R^N)$ and the semigroup of operators $(\Heat_t)_{t\ge0}$ is given by 
\begin{equation*}
(\Heat_t u,v)_{L^2}
=
\int_{\R^N}e^{-\sus(\xi)t}\,\hat u(\xi)\cdot\overline{\hat v(\xi)}\di\xi
\end{equation*}
for all $t\ge0$ and $u,v\in L^2(\R^N)$, where $\sus\colon\R^N\to[0,\infty]$ is a measurable function (in the aforementioned cases, $\sus(\xi)=(2\pi|\xi|)^{2s}$ for $\xi\in\R^N$ and $s\in(0,1]$), then a different and simpler approach via Fourier transform is also possible, see \cref{subsec:fourier} for more details.

\subsection{Organization of the paper}
The rest of the paper is organized as follows.
In \cref{sec:preliminaries}, we provide the main notation and the basic results used throughout the paper.
In \cref{sec:bbm_sharp}, we deal with the sharp conditions for the BBM formula in \cref{resi:bbm}.
In \cref{sec:special}, we specialize the BBM formula to the family of kernels~\eqref{eqi:rho_special} and we prove the compactness criterion stated in \cref{resi:compactness}.
In \cref{sec:non-local}, we treat non-local-to-non-local results, proving \cref{resi:diretto_lim,resi:diretto_comp}.
In \cref{sec:heat}, we apply our theorems to the study of energies induced by heat-type kernels, both in the local and in the non-local setting, proving \cref{resi:heat,resi:frac_heat}.
Finally, in \cref{sec:hilbert}, we detail the proof of \cref{resi:hilbert} and present an alternative proof of heat content asymptotics in $L^2(\R^N)$ via Fourier transform.

\section{Preliminaries}

\label{sec:preliminaries}

\subsection{General notation}

We let $N\in\N$ and $\mathbb S^{N-1}=\set*{x\in\R^N:|x|=1}$ be the $(N-1)$-dimensional unit sphere in $\R^N$.

We let $\mathscr L^N$ be the $N$-dimensional Lebesgue measure in~$\R^N$ and we let $\mathscr H^s$ be the $s$-dimensional Hausdorff measure in~$\R^N$, with $s\in[0,N]$.
All sets and functions are assumed to be Lebesgue measurable.
We use the shorthand $|E|=\mathscr L^N(E)$ for $E\subset\R^N$.

Given a non-empty set $X\subset\R^N$, we let $C(X)$ and $\operatorname{Lip}(X)$ be the spaces of continuous and Lipschitz continuous functions on~$X$, respectively. 
As customary, we let $C_c(X)$ and $\operatorname{Lip}_c(X)$ be their subsets of compactly supported functions, respectively.
If $X$ is open, then we also let $C^\infty(X)$ and $C^\infty_c(X)$ be the spaces of smooth functions and of smooth functions with compact support on~$X$, respectively.

\subsection{Radon measures}

Let $X\subset\R^N$ be a non-empty set. 
We let $\mathscr M(X)$ and $\mathscr M_{\rm loc}(X)$ be the spaces of finite and locally finite signed Radon measures on~$X$, respectively.
We also let $\mathscr M^+(X)$ and $\mathscr M_{\rm loc}^+(X)$ be their subsets of non-negative measures, respectively.

By the Riesz Representation Theorem, $\mathscr M_{\rm loc}(X)$ can be identified as the dual of $C_c(X)$, endowed with local uniform convergence.
Thus, we say that $(\mu_k)_{k\in\N}\subset\mathscr M_{\rm loc}(X)$ converges to $\mu\in\mathscr M_{\rm loc}(X)$ in the (\emph{local}) \emph{weak$^\star$ sense}, and we write $\mu_k\weakstarto\mu$ in $\mathscr M_{\rm loc}(X)$ as $k\to\infty$, if 
\begin{equation}
\label{eq:weakstar}
\lim_{k\to\infty}
\int_Xf\di\mu_k
=
\int_Xf\di\mu
\quad
\text{for every}\ 
f\in C_c(X).
\end{equation}

We recall that, if $\mu_k\weakstarto\mu$ in $\mathscr M_{\rm loc}(X)$ as $k\to\infty$, then~\eqref{eq:weakstar} actually holds for every bounded Borel function $f\colon X\to\R$ with compact support such that the set of its discontinuity points is $\mu$-negligible.
Consequently, if $X$ is compact and $\mu_k\weakstarto\mu$ in $\mathscr M_{\rm loc}(X)$ as $k\to\infty$, then~\eqref{eq:weakstar} holds for every $f\in C(X)$.
Therefore, in this case, we simply write $\mu_k\weakstarto\mu$ in $\mathscr M(X)$ as $k\to\infty$.
See~\cite{AFP00} for a more detailed discussion.

For future convenience, we recall the following result, which corresponds to~\cite{P04a}*{Lem.~6}.

\begin{lemma}
\label{res:wei}
Let $\mu\in\mathscr M^+(\mathbb S^{N-1})$ and let $\Theta_\mu\in C(\R^N)$ be defined as
\begin{equation}
\label{eq:wei}
\Theta_\mu(v)
=
\int_{\mathbb S^{N-1}}|v\cdot\sigma|\di\mu(\sigma),
\quad
\text{for every}\ v\in\R^N.
\end{equation}
Then, $\min_{\mathbb S^{N-1}}\Theta_\mu>0$ if and only if $\operatorname{span}(\operatorname{supp}\mu)=\R^N$. 
\end{lemma}

\subsection{Sobolev and \texorpdfstring{$BV$}{BV} spaces}

For $p\in[1,\infty)$, we let
\begin{equation*}
{\sob{p}}(\R^N)
=
\begin{cases}
W^{1,p}(\R^N)
&
\text{for}\ p>1
\\[1ex]
BV(\R^N)
&
\text{for}\ p=1.
\end{cases}
\end{equation*}
As customary, $Du$ denotes the distributional gradient of $u\in {\sob{p}}(\R^N)$.
In particular, if $p=1$, then $Du$ may be a finite Radon measure on $\R^N$.
We endow ${\sob{p}}(\R^N)$ with the norm
\begin{equation*}
\|u\|_{{\sob{p}}(\R^N)}
=
\left(\|u\|_{L^p}^p+\|Du\|_{L^p}^p\right)^{1/p},
\quad
\text{for}\
u\in {\sob{p}}(\R^N),
\end{equation*}
where, as customary, we have set
\begin{equation}
\label{eq:p_seminorm}
\|Du\|_{L^p}^p
=
\begin{cases}
\displaystyle
\int_{\R^N}|Du(x)|^p\di x
&
\text{for}\
p>1,
\\[3ex]
|Du|(\R^N)
&
\text{for}\ p=1.
\end{cases}
\end{equation}

We recall the following simple result, whose proof is omitted.

\begin{lemma}
\label{res:ftc}
Let $p\in[1,\infty)$ and $z\in\R^N$.
The following hold:

\begin{enumerate}[label=(\roman*),itemsep=1ex,topsep=1ex]

\item
\label{item:ftc1} 
if $u\in\sob{p}(\R^N)$, then $\|u(\,\cdot+z)-u\|_{L^p}\le\|z\cdot Du\|_{L^p}$;

\item
\label{item:ftc2}
if $u\in W^{2,p}(\R^N)$, then $
\left|
\|u(\,\cdot+z)-u\|_{L^p}
-
\|z\cdot Du\|_{L^p}
\right|
\le 
\frac{|z|^2}2\,\|D^2u\|_{L^p}$.

\end{enumerate}
\end{lemma}

\subsection{Non-local Sobolev spaces}

Given $p\in[1,\infty)$ and a non-negative measurable function $\kappa\colon\R^N\to[0,\infty]$, we let 
\begin{equation*}
W^{\kappa,p}(\R^N)
=
\set*{u\in L^p(\R^N) : [u]_{W^{\kappa,p}}<\infty},
\end{equation*}
where
\begin{equation*}
[u]_{W^{\kappa,p}}
=
\left(\int_{\R^N}\|u(\cdot+z)-u\|_{L^p}^p\,\kappa(z)\di z\right)^{1/p}
\end{equation*}
for $u\in L^p(\R^N)$. 
We note that $W^{\kappa,p}(\R^N)$, endowed with the norm
\begin{equation*}
\|u\|_{W^{\kappa,p}}
=
\left(
\|u\|_{L^p}^p+[u]_{W^{\kappa,p}}^p\right)^{1/p},
\quad
\text{for}\
u\in W^{\kappa,p}(\R^N),
\end{equation*}
is a Banach space. We refer to~\cites{BS24,F25} for a more detailed presentation.
Here we only mention that the \emph{fractional Sobolev--Slobodeckij space} $W^{s,p}(\R^N)$, with $s\in(0,1)$ and $p\in[1,\infty)$, corresponds to the choice $\kappa(z)=|z|^{-N-sp}$ for all $z\in\R^N\setminus\set*{0}$.

\subsection{Convergence of functionals on \texorpdfstring{$L^p(\R^N)$}{Sp}}

Let $p\in[1,\infty)$ and let $\mathcal X^p(\R^N)\subset L^p(\R^N)$.
Given $F_k,G\colon L^p(\R^N)\to[0,\infty]$, $k\in\N$, we adopt the following terminology.
For a complete description of \emph{$\Gamma$-convergence}, we refer to the monographs~\cites{Braides02,DalMaso93}.

\begin{definition}[Pointwise convergence]
\label{def:pointwise_conv}
We say that $(F_k)_{k\in\N}$ converges to $G$ on $\mathcal X^p(\R^N)$ as $k\to\infty$ in the \textit{pointwise sense} if $\displaystyle\lim_{k\to\infty}F_k(u)=G(u)$ for every $u\in \mathcal X^p(\R^N)$.
\end{definition} 

\begin{definition}[$\Gamma$-convergence]
\label{def:gamma_conv}
We say that $(F_k)_{k\in\N}$ converges to $G$ on $\mathcal X^p(\R^N)$ as $k\to\infty$ in the \emph{$\Gamma$-sense with respect to the $L^p$ topology} if the following two properties hold:
\begin{enumerate}[label=$\bullet$,itemsep=.5ex,leftmargin=2em]

\item 
($\Gamma$-$\liminf$)
if $(u_k)_{k\in\N}\subset L^p(\R^N)$ is such that $u_k\to u$ in $L^p(\R^N)$ as $k\to\infty$ for some $u\in \mathcal X^p(\R^N)$, then
$\displaystyle\liminf_{k\to\infty} F_k(u_k)
\ge 
G(u)$;

\item
($\Gamma$-$\limsup$)
if $u\in \mathcal X^p(\R^N)$, then there exists $(u_k)_{k\in\N}\subset \mathcal X^p(\R^N)$ such that $u_k\to u$ in $L^p(\R^N)$ as $k\to\infty$ and 
$\displaystyle\limsup_{k\to\infty}F_k(u_k)\le G(u)$. 

\end{enumerate}
\end{definition}

\begin{definition}[Coerciveness]
\label{def:coercive}
We say that $(F_k)_{k\in\N}$ is \emph{coercive} on $\mathcal X^p(\R^N)$ if, whenever $(u_k)_{k\in\N}\subset L^p(\R^N)$ is such that $u_k\to u$ for some $u\in L^p(\R^N)$ as $k\to\infty$ and 
\begin{equation*}
\liminf_{k\to\infty} F_k(u_k)<\infty,    
\end{equation*}
then $u\in \mathcal X^p(\R^N)$.
\end{definition}

\subsection{Family of kernels}

Throughout the paper, we let $I=(0,1)$.
We let $(\rho_t)_{t\in I}\subset L^1_{\rm loc}(\R^N)$ be a family of non-negative kernels, $\rho_t\ge0$ for every $t\in I$.
The following result generalizes~\cite{DDP24}*{Lem.~5.2}. 
We briefly detail its proof for the convenience of the reader.

\begin{lemma}
\label{res:diavolacci}
Let $p\in[1,\infty)$ and $J\subset I$ be such that $0\in\overline J$.
The following are equivalent:
\begin{enumerate}[label=(\roman*),itemsep=1ex,topsep=1ex]

\item 
\label{item:diavolacci_insieme}
$\displaystyle\sup_{R>0}
\limsup_{t\in J,\,t\to0^+}
R^p\int_{\R^N}\frac{\rho_t(z)}{R^p+|z|^p}\di z<\infty$;

\item
\label{item:diavolacci_staccati}
$\displaystyle
\sup_{R>0}
\left[
\limsup_{t\in J,\,t\to0^+}
\int_{B_R}\rho_t(z)\di z
+
\limsup_{t\in J,\,t\to0^+}
R^p\int_{B_R^c}\frac{\rho_t(z)}{|z|^p}\di z
\right]<\infty$;

\item
\label{item:diavolacci_R0}
there exists $R_0>0$ such that  
\begin{equation*}
\limsup_{t\in J,\,t\to0^+}\int_{B_{R_0}}\rho_t(z)\di z
+
\sup_{R>R_0}
\limsup_{t\in J,\,t\to0^+}
R^p\int_{B_R^c}\frac{\rho_t(z)}{|z|^p}\di z<\infty;
\end{equation*}

\item
\label{item:diavolacci_zero}
$\displaystyle\sup_{R>0}\limsup_{t\in J,\,t\to0^+}\int_{B_R}\rho_t(z)\di z<\infty$
and
\begin{equation*}
\limsup_{t\in J,\,t\to0^+}\int_{B_R^c}\frac{\rho_t(z)}{|z|^p}\di z=0
\quad
\text{for all}\ R>0;
\end{equation*} 

\item
\label{item:diavolacci_foghem}
$\displaystyle\limsup_{t\in J,\,t\to0^+}\int_{\R^N}(1\wedge|z|^{-p})\,\rho_t(z)\di z<\infty$
and
\begin{equation*}
\limsup_{t\in J,\,t\to0^+}\int_{B_R^c}(1\wedge|z|^{-p})\,\rho_t(z)\di z=0
\quad
\text{for all}\ R>0.
\end{equation*}

\end{enumerate}
\end{lemma}

\begin{proof}
The equivalence \ref{item:diavolacci_insieme}$\iff$\ref{item:diavolacci_staccati} can be proved  \textit{verbatim} as in~\cite{DDP24}*{Lem.~5.2}.
It is clear that \ref{item:diavolacci_staccati}$\implies$\ref{item:diavolacci_R0}, while the fact that
\begin{equation*}
\int_{B_R}\rho_t(z)\di z
= 
\left(\int_{B_{R_0}}+
\int_{B_R\setminus B_{R_0}}\right)
\rho_t(z)\di z
\le 
\int_{B_{R_0}}\rho_t(z)\di z
+
R^p
\int_{B_R}
\frac{\rho_t(z)}{|z|^p}\di z
\end{equation*}
for every $R>R_0$ yields that \ref{item:diavolacci_R0}$\implies$\ref{item:diavolacci_staccati}.
The implication \ref{item:diavolacci_staccati}$\implies$\ref{item:diavolacci_zero} is obvious.
Conversely, given $R>0$, by the first part of~\ref{item:diavolacci_zero} we can find $M>R$ such that 
\begin{equation*}
\limsup_{t\in J,\,t\to0^+}\int_{B_M^c}\frac{\rho_t(z)}{|z|^p}\di z
\le 
\frac1{R^p},
\end{equation*} 
so that
\begin{equation*}
\begin{split}
\limsup_{t\in J,\,t\to0^+}\int_{B_R^c}\frac{\rho_t(z)}{|z|^p}\di z
&\le 
\limsup_{t\in J,\,t\to0^+}\int_{B_M^c}\frac{\rho_t(z)}{|z|^p}\di z
+
\limsup_{t\in J,\,t\to0^+}\int_{B_M\setminus B_R}\frac{\rho_t(z)}{|z|^p}\di z
\\
&\le
\frac1{R^p}
+
\frac1{R^p}
\limsup_{t\in J,\,t\to0^+}\int_{B_M\setminus B_R}\rho_t(z)\di z
\le
\frac{C+1}{R^p},
\end{split}
\end{equation*}
where 
\begin{equation*}
C\coloneqq\sup_{r>0}\limsup_{t\in J,\,t\to0^+}\int_{B_r}\rho_t(z)\di z<\infty,    
\end{equation*}
proving that \ref{item:diavolacci_zero}$\implies$\ref{item:diavolacci_staccati}.
The equivalence \ref{item:diavolacci_foghem}$\iff$\ref{item:diavolacci_R0} follows via elementary arguments, so we omit its proof.
\end{proof}

\begin{remark}
The equivalence \ref{item:diavolacci_insieme}$\iff$\ref{item:diavolacci_staccati} in \cref{res:diavolacci} is proved in~\cite{DDP24}*{Lem.~5.2} for $p=2$.
Property~\ref{item:diavolacci_foghem} in \cref{res:diavolacci} is exploited in~\cite{F25} for radially symmetric families $(\rho_t)_{t\in I}$ and $p>1$, see~\cite{F25}*{Sec.~9.2} for more details.
A family complying with any of the conditions in \cref{res:diavolacci} is given by the standard fractional kernels~\cites{BBM01,P04a}, defined as  
\begin{equation*}
\rho_t(z)=\frac{1-t}{|z|^{N-(1-t)p}},
\end{equation*}
for $z\in\R^N\setminus\set*{0}$, $t\in I$, and $p\in[1,\infty)$.
\end{remark}

Following~\cite{P04a}*{Sec.~1.3} and~\cite{LS11}*{Sec.~1}, we introduce the following terminology.
From now on, given $\tau\in I$ and $v\in\mathbb S^{N-1}$, we set 
\begin{equation*}
\mathcal C_\tau(v)
=
\set*{x\in\R^N: x\cdot v\ge(1-\tau)\,|x|}.
\end{equation*}

\begin{definition}[Maximal rank]
\label{def:maxrank}
Let $J\subset I$ be such that $0\in\overline J$.
We say that the family $(\rho_t)_{t\in J}\subset L^1_{\rm loc}(\R^N)$ has \emph{maximal rank} if there exist $\tau\in I$ and a basis $v_1,\dots, v_N\in\mathbb S^{N-1}$ of~$\R^N$ such that $\mathcal C_\tau(v_i)\cap\mathcal C_\tau(v_j)=\set*{0}$ for every $i,j\in\set*{1,\dots,N}$ such that $i\ne j$, and 
\begin{equation*}
\inf_{\delta>0}
\,
\min_{i\in\set*{1,\dots,N}}
\,
\liminf_{t\in J,\,t\to0^+}
\int_{B_\delta\,\cap\,\mathcal C_\tau(v_i)}\rho_t(z)\di z>0.
\end{equation*}
\end{definition}

This maximal-rank condition is the nonlocal analogue of \emph{uniform ellipticity}, guaranteeing full $N$-dimensional control of all (nonlocal) partial derivatives: the family of kernels is forced to retain a uniform amount of mass in narrow cones around every basis direction, so it never collapses onto a lower-dimensional subspace.

The (first part of the) following result was implicitly proved across the proof of~\cite{DDP24}*{Prop.~3.2}.
For similar results, we refer to~\cites{P04a} and~\cite{LS11}*{Lem.~2.1}.

\begin{lemma}
\label{res:munu_limit}
Let $J\subset I$ be such that $0\in\overline J$.
If
\begin{equation*}
\sup_{R>0}\limsup_{t\in J,\,t\to0^+}\int_{B_R}\rho_t(z)\di z<\infty,
\end{equation*}
then there exist a countable set $I_0\subset I$, two infinitesimal sequences $(t_k)_{k\in\N}\subset J$ and $(\delta_l)_{l\in\N}\subset I\setminus I_0$, and two measures $\mu\in\mathscr M^+(\mathbb S^{N-1})$ and $\nu\in\mathscr M^+(\R^N)$ such that:

\begin{enumerate}[label=(\roman*),itemsep=1ex,topsep=1ex]

\item
\label{item:munu_limit1}
letting $\nu_k=\rho_{t_k}\,\mathscr L^N$ for every $k\in\N$, it holds that
$\nu_k\weakstarto\nu$ in $\mathscr M_{\rm loc}(\R^N)$ as $k\to\infty$;

\item
\label{item:munu_limit2}
letting $A_l=\set*{x\in\R^N : \delta_l<|x|<\frac1{\delta_l}}$ for every $l\in\N$, it holds that 
\begin{equation*}
\lim_{k\to\infty}
\int_{A_l} f\di\nu_k
=
\int_{A_l}f\di\nu,
\quad
\text{for every $l\in\N$ and $f\in C(\R^N\setminus\set*{0})$;}
\end{equation*}

\item
\label{item:munu_limit3}
letting $\mu_t^\delta\in\mathscr M^+(\mathbb S^{N-1})$ be given by 
\begin{equation}
\label{eq:spherical_meas}
\mu_t^\delta(E)
=
\int_E\left(\int_0^\delta\rho_t(\sigma r)\,r^{N-1}\di r\right)\di\mathscr H^{N-1}(\sigma)
\end{equation}
for every $\mathscr H^{N-1}$-measurable set $E\subset\mathbb S^{N-1}$ and $t,\delta\in I$, it holds that:
\begin{enumerate}[itemsep=1ex,topsep=1ex]

\item
letting $\mu_k^l=\mu_{t_k}^{\delta_l}$ for every $k,l\in\N$, there exists $(\mu^l)_{l\in\N}\subset\mathscr M^+(\mathbb S^{N-1})$ such that 
$\mu_k^l\weakstarto \mu^l$ in $\mathscr M(\mathbb S^{N-1})$ as $k\to\infty$ for every $l\in\N$;

\item
the sequence $(\mu^l)_{l\in\N}\subset\mathscr M^+(\mathbb S^{N-1})$ is such that 
$\mu^l\weakstarto\mu$ in $\mathscr M(\mathbb S^{N-1})$ as $l\to\infty$.

\end{enumerate}   
\end{enumerate}

% FINO A QUI

In addition, if $(\rho_t)_{t\in J}$ has maximal rank, then $\operatorname{span}(\operatorname{supp}\mu)=\R^N$. 
\end{lemma}

In the proof of \cref{res:munu_limit}, we will need the following simple result, which can be inferred as in the proof of~\cite{LS11}*{Lem.~2.1}.
We thus omit its proof.

\begin{lemma}
\label{res:leospec}
Let $\tau\in I$ and let $v_1,\dots,v_N\in\mathbb S^{N-1}$ be a basis of $\R^N$ such that 
$\mathcal C_\tau(v_i)\cap\mathcal C_\tau(v_j)=\set{0}$
for every $i,j\in\set*{1,\dots,N}$ such that $i\ne j$.
Then, there exists $c_0>0$ such that, for each $v\in\mathbb S^{N-1}$, there exists $i_0\in\set*{1,\dots,N}$ such that $|v\cdot\sigma|\ge c_0$ for every $\sigma\in\mathcal C_\tau(v_{i_0})\cap\mathbb S^{N-1}$.
\end{lemma}

\begin{proof}[Proof of \cref{res:munu_limit}]
Let us set 
\begin{equation}
\label{eq:munu_limit_ass}
M=\sup_{R>0}\limsup_{t\in J,\,t\to0^+}\int_{B_R}\rho_t(z)\di z<\infty.
\end{equation}
We split the proof in three steps.

\vspace{1ex}

\textit{Step 1}.
Let $\nu_t=\rho_t\,\mathscr L^N$ for all $t\in J$.
Let $(R_k)_{k\in\N}\subset(0,\infty)$ be a strictly increasing and unbounded sequence. Owing to~\eqref{eq:munu_limit_ass}, we can find a strictly decreasing and infinitesimal sequence $(t_k)_{k\in\N}\subset J$ such that
\begin{equation}
\label{eq:mucca}
\nu_{t_k}(B_{R_k})
=
\int_{B_{R_k}}\rho_{t_k}(z)\di z
\le 
M+\frac1k,
\end{equation}
for each $k\in\N$.
By known results (see~\cite{AFP00}*{Th.~1.59} for example), we can find $\nu\in\mathscr M^+(\R^N)$ such that, up to passing to a non-relabeled subsequence of $(t_k)_{k\in\N}$, $\nu_{t_k}\weakstarto\nu$ in $\mathscr M_{\rm loc}(\R^N)$ as $k\to\infty$, and $\nu(\R^N)\le M+1$.

\vspace{1ex}

\textit{Step 2}.
Let us set $A_\delta=\set*{x\in\R^N : \delta<|x|<\frac1\delta}$ for all $\delta\in I$.
By definition, the family $(\partial A_\delta)_{\delta\in I}$ consists of pairwise disjoint compact sets in $\R^N$ which are precisely the sets of discontinuity points of the functions $(\chi_{A_\delta})_{\delta\in I}$.
We can thus apply~\cite{AFP00}*{Prop.~1.62(b)} (see also the discussion in~\cite{AFP00}*{Ex.~1.63}) and find a countable set $I_0\subset I$ such that $\nu(\partial A_\delta)=0$ for all $\delta\in I\setminus I_0$, where $\nu$ is the measure obtained in Step~1.
Since $\nu_{t_k}\weakstarto\nu$ in $\mathscr M_{\rm loc}(\R^N)$ as $k\to\infty$ by Step~1 and $\bigcup_{\delta\in I}A_\delta=\R^N\setminus\set*{0}$ by construction, we hence infer that 
\begin{equation*}
\lim_{k\to\infty}
\int_{A_\delta}f\di\nu_k
=
\int_{A_\delta}f\di\nu
\quad
\text{for every $\delta\in I\setminus I_0$ and $f\in C(\R^N\setminus\set*{0})$}.
\end{equation*}

\vspace{1ex}

\textit{Step~3}.
Now pick any strictly decreasing sequence $(\delta_l)_{l\in\N}\subset I\setminus I_0$ such that $\delta_l\to0^+$ as $l\to\infty$ and consider the measures $(\mu_{t_k}^{\delta_l})_{l,k\in\N}$ as defined in~\eqref{eq:spherical_meas}, where $(t_k)_{k\in\N}$ is as in Step~1.
Owing to~\eqref{eq:spherical_meas} and to~\eqref{eq:mucca}, we get that 
\begin{equation*}
\mu_{t_k}^{\delta_l}(\mathbb S^{N-1})
\le
\int_{\mathbb S^{N-1}}
\left(
\int_0^1 \rho_{t_k}(\sigma r)\,r^{N-1}\di r
\right)
\di\mathscr H^{N-1}(\sigma)
=
\int_{B_1}\rho_{t_k}(z)\di z
\le
M+1 
\end{equation*}
for each $k\in\N$ sufficiently large and each $l\in\N$. 
Thus, for each $l\in\N$, we can find a subsequence $(t_k^l)_{k\in\N}$ of $(t_k)_{k\in\N}$ and  $\mu^l\in\mathscr M^+(\mathbb S^{N-1})$ such that $\mu_{t_k^l}^{\delta_l}\weakstarto\mu^l$ in $\mathscr M(\mathbb S^{N-1})$ as $k\to\infty$ and $\mu^l(\mathbb S^{N-1})\le M+1$.
By a routine diagonal argument, we can find subsequence $(t_k')_{k\in\N}$ of $(t_k)_{k\in\N}$ such that $\mu_{t_k'}^{\delta_l}\weakstarto\mu^l$ in $\mathscr M(\mathbb S^{N-1})$ as $k\to\infty$ for all $l\in\N$. 
Moreover,  we can find a subsequence $(\delta_{l_j})_{j\in\N}$ of $(\delta_l)_{l\in\N}$ and $\mu\in\mathscr M^+(\mathbb S^{N-1})$ such that $\mu^{l_j}\weakstarto\mu$ in $\mathscr M(\mathbb S^{N-1})$ as $j\to\infty$ and $\mu(\mathbb S^{N-1})\le M+1$.

\vspace{1ex}

Combining Steps~1, 2 and~3 above, we infer the validity of~\ref{item:munu_limit1}, \ref{item:munu_limit2} and~\ref{item:munu_limit3}, concluding the proof of the first part of the statement of \cref{res:munu_limit}.

\vspace{1ex}

\textit{Step~4}.
Let us now additionally assume that $(\rho_t)_{t\in J}$ has maximal rank as in \cref{def:maxrank}.
Thus, there exist $\tau\in I$ and a basis $v_1,\dots, v_N\in\mathbb S^{N-1}$ of~$\R^N$ such that $\mathcal C_\tau(v_i)\cap\mathcal C_\tau(v_j)=\set*{0}$ for every $i,j\in\set*{1,\dots,N}$ such that $i\ne j$, and 
\begin{equation}
\label{eq:massrang}
m\coloneqq
\inf_{\delta>0}
\,
\min_{i\in\set*{1,\dots,N}}
\,
\liminf_{t\in J,\,t\to0^+}
\int_{B_\delta\,\cap\,\mathcal C_\tau(v_i)}\rho_t(z)\di z>0.
\end{equation}
Therefore, we can apply \cref{res:leospec}, so we let $c_0>0$ be given by \cref{res:leospec}.
Let us now fix $v\in\mathbb S^{N-1}$.
Again by \cref{res:leospec},  there exists $i_0\in\set*{1,\dots,N}$ such that $|v\cdot\sigma|\ge c_0$ for every $\sigma\in\mathcal C_\tau(v_{i_0})\cap\mathbb S^{N-1}$.
Now let $(\delta_{l_j})_{j\in\N}$ and $(t_k')_{k\in\N}$ be the sequences defined in Step~3.
Recalling~\eqref{eq:wei} and setting $\mu_k^j\coloneqq \mu_{t_k'}^{\delta_{l_j}}$ for every $k,j\in\N$, we can write
\begin{equation*}
\begin{split}
\Theta_{\mu_k^j}(v)
&=
\int_{\mathbb S^{N-1}}|v\cdot\sigma|\di\mu_k^j(\sigma)
\ge 
\int_{\mathbb S^{N-1}\,\cap\,\mathcal C_\tau(v_{i_0})}|v\cdot\sigma|\di\mu_k^j(\sigma)
\ge 
c_0\,
\mu_j^k
\left(
\mathbb S^{N-1}
\cap
\mathcal C_\tau(v_{i_0})\right)
\\
&=
c_0
\int_{B_{\delta_{l_j}}\,\cap\,\mathcal C_\tau(v_{i_0})}
\rho_{t_k'}(z)\di z
\end{split}
\end{equation*}
for every $k,j\in\N$.
On the one side, recalling~\eqref{eq:massrang}, we can estimate
\begin{equation*}
\begin{split}
\liminf_{k\to\infty}
\int_{B_{\delta_{l_j}}\,\cap\,\mathcal C_\tau(v_{i_0})}
\rho_{t_k'}(z)
\di z
&\ge 
\liminf_{t\in J,\,t\to0^+}
\int_{B_{\delta_{l_j}}\,\cap\,\mathcal C_\tau(v_{i_0})}
\rho_t(z)\di z
\\
&\ge 
\min_{i\in\set*{1,\dots,N}}
\,
\liminf_{t\in J,\,t\to0^+}
\int_{B_{\delta_{l_j}}\,\cap\,\mathcal C_\tau(v_i)}
\rho_t(z)\di z
\\
&
\ge
\inf_{\delta>0}
\,
\min_{i\in\set*{1,\dots,N}}
\,
\liminf_{t\in J,\,t\to0^+}
\int_{B_\delta\,\cap\,\mathcal C_\tau(v_i)}
\rho_t(z)\di z
=m,
\end{split}
\end{equation*}
while, on the other hand, by Step~3, we have
\begin{equation*}
\liminf_{k\to\infty}
\Theta_{\mu_k^j}(v)
=
\lim_{k\to\infty}
\int_{\mathbb S^{N-1}}|v\cdot\sigma|\di\mu_k^j(\sigma)
=
\int_{\mathbb S^{N-1}}|v\cdot\sigma|\di\mu^{l_j}(\sigma)
=
\Theta_{\mu^{l_j}}(v)
\end{equation*}
for every $j\in\N$.
We hence get that $\Theta_{\mu^{l_j}}(v)\ge c_0\,m$ for every $v\in\mathbb S^{N-1}$ and $j\in\N$.
Again by Step~3, passing to the limit as $j\to\infty$, we deduce that $\Theta_\mu(v)\ge c_0\,m$ for every $v\in\mathbb S^{N-1}$; that is, $\inf_{\mathbb S^{N-1}}\Theta_\mu>0$.
The conclusion hence follows from \cref{res:wei}. 
\end{proof}

\subsection{The \texorpdfstring{functionals $\mathscr F_{t,p}$}{BBM functionals}}

Let $p\in[1,\infty)$.
We define $\mathscr F_{t,p}\colon L^p(\R^N)\to[0,\infty]$ by letting 
\begin{equation}
\label{eq:F_energy}
\mathscr F_{t,p}(u)
=
\int_{\R^N}\frac{\|u(\,\cdot+z)-u\|_{L^p}^p}{|z|^p}\,\rho_t(z)\di z
\end{equation}
for $u\in L^p(\R^N)$ and $t\in I$.
The following result improves and generalizes~\cite{DDP24}*{Prop.~4.1}.

\begin{lemma}
\label{res:pulce}
Let $J\subset I$ be such that $0\in\overline J$.
If there exists $C>0$ such that  
\begin{equation}
\label{eq:pulce}
\limsup_{t\in J,\,t\to0^+}
\mathscr F_{t,p}(u)
\le 
C\|Du\|_{L^p}^p
\end{equation}
for every $u\in\operatorname{Lip}_c(\R^N)$,
then any of the properties in \cref{res:diavolacci} holds.
\end{lemma}

\begin{proof}
Let $R>0$ and define $u_R\in\operatorname{Lip}_c(\R^N)$ by letting
\begin{equation*}
u_R(x)
=
\begin{cases}
1 
& 
\text{for}\ x\in B_{R/4},
\\[1ex]
2-\frac 4R\,|x|
& 
\text{for}\ x\in B_{R/2}\setminus B_{R/4},
\\[1ex]
0
&
\text{for}\ x\in B_{R/2}^c.
\end{cases}
\end{equation*}
Since $\operatorname{supp} u_R\subset B_{R/2}$, we have that 
\begin{equation*}
\|u_R(\,\cdot+z)-u_R\|_{L^p}^p=2\|u_R\|_{L^p}^p
\ge 
2|B_{R/4}|
\end{equation*}
for every $z\in B_R^c$, from which we deduce that 
\begin{equation}
\label{eq:pane}
\mathscr F_{t,p}(u_R)
\ge 
\int_{B_R^c}\frac{\|u_R(\,\cdot+z)-u_R\|_{L^p}^p}{|z|^p}\,\rho_{t}(z)\di z
\ge 
2|B_{R/4}|
\int_{B_R^c}
\frac{\rho_t(z)}{|z|^p}\di z.
\end{equation}
On the other hand, we can estimate
\begin{equation}
\label{eq:salame}
\|Du_R\|_{L^p}^p
=
\int_{B_{R/2}\setminus B_{R/4}}|Du_R(x)|^p\di x
\le 
\left(\frac 4R\right)^p\,|B_{R/2}|.
\end{equation}
By combining~\eqref{eq:pulce} with~\eqref{eq:pane} and~\eqref{eq:salame}, we infer that 
\begin{equation}
\label{eq:billy}
\limsup_{t\in J,\,t\to0^+}
\int_{B_R^c}
\frac{\rho_t(z)}{|z|^p}\di z
\le 
C\,
\left(\frac 4R\right)^p\,\frac{|B_{R/2}|}{2|B_{R/4}|}
=
C\,2^{N-1+2p}\,R^{-p},
\end{equation}
for $R>0$.
In view of~\ref{item:diavolacci_R0} in \cref{res:diavolacci}, to conclude, we just need to prove that 
\begin{equation}
\label{eq:tacos}
\limsup_{t\in J,\,t\to0^+}
\int_{B_1}\rho_t(z)\di z<\infty.
\end{equation}
To this aim, we define $v\in\operatorname{Lip}_c(\R^N)$ by letting
\begin{equation*}
v(x)
=
\begin{cases}
e^{|x|^2} 
& 
\text{for}\ x\in B_2,
\\[1ex]
e^4(3-|x|)
& 
\text{for}\ x\in B_3\setminus B_2,
\\[1ex]
0
& 
\text{for}\ x\in B_3^c.
\end{cases}
\end{equation*}
Since $v$ is radially symmetric, a change of variable yields  
\begin{equation*}
\|v(\,\cdot+z)-v\|_{L^p}^p
=
\|v(\,\cdot+|z|\mathrm e_1)-v\|_{L^p}^p
\end{equation*}
for every $z\in\R^N$.
Therefore, letting $D=\set*{x\in B_1:x_1\ge\frac12}\subset B_1$, we can estimate
\begin{equation*}
\begin{split}
\|v(\,\cdot+z)-v\|_{L^p}^p
&\ge 
\int_D\left|e^{|x+|z|\mathrm e_1|^2}-e^{|x|^2}\right|^p\di x
=
\int_D e^{p|x|^2}\left|e^{|z|^2+2|z|x_1}-1\right|^p\di x
\\
&\ge 
\int_D \left(e^{|z|}-1\right)^p\di x
\ge
|D|\,|z|^p
\end{split}
\end{equation*}
for every $z\in B_1$, from which we deduce that 
\begin{equation}
\label{eq:formaggio}
\mathscr F_{t,p}(v)
\ge 
\int_{B_1}
\frac{\|v(\,\cdot+z)-v\|_{L^p}^p}{|z|^p}\,\rho_{t}(z)\di z
\ge 
|D|
\int_{B_1}\rho_{t}(z)\di z.
\end{equation}
By combining~\eqref{eq:formaggio} with~\eqref{eq:pulce}, we infer~\eqref{eq:tacos} and thus conclude the proof.
\end{proof}

\subsection{The  \texorpdfstring{functional $\mathscr G^{\mu,\nu}_p$}{limit functional}}

Given two measures $\mu\in\mathscr M^+(\mathbb S^{N-1})$ and $\nu\in\mathscr M^+(\R^N)$, we define  $\mathscr G^{\mu,\nu}_p\colon L^p(\R^N)\to[0,\infty)$ by letting 
\begin{equation}
\label{eq:G}
\mathscr G^{\mu,\nu}_p(u)
=
\begin{cases}
\displaystyle\int_{\mathbb S^{N-1}}
\|\sigma\cdot Du\|_{L^p}^p\di\mu(\sigma)
+
\int_{\R^N\setminus\set{0}}
\frac{\|u(\cdot+z)-u\|_{L^p}^p}{|z|^p}\di \nu(z)
&
\text{for}\ u\in {\sob{p}}(\R^N),
\\[2ex]
\infty
&
\text{otherwise}.
\end{cases}
\end{equation}
Here, as in~\eqref{eq:p_seminorm}, for every $\sigma\in\mathbb S^{N-1}$ and $u\in {\sob{p}}(\R^N)$, we have set
\begin{equation*}
\|\sigma\cdot Du\|_{L^p}^p
=
\begin{cases}
\displaystyle
\int_{\R^N}|\sigma\cdot Du(x)|^p\di x
&
\text{for}\
p>1,
\\[3ex]
|\sigma\cdot Du|(\R^N)
&
\text{for}\
p=1.
\end{cases}
\end{equation*}
We observe that, for each fixed $\sigma\in\mathbb S^{N-1}$, the energy $u\mapsto\|\sigma\cdot Du\|_{L^p}$ is lower semicontinuous on ${\sob{p}}(\R^N)$ with respect to the $L^p$ convergence of functions.

We collect the basic properties of $\mathscr G_p^{\mu,\nu}$ in the following result, whose proof is omitted.

\begin{lemma}
Let $p\in[1,\infty)$, $\mu\in\mathscr M^+(\mathbb S^{N-1})$ and $\nu\in\mathscr M^+(\R^N)$.
The functional $\mathscr G_p^{\mu,\nu}$ satisfies the bound  
\begin{equation}
\label{eq:G_torta}
\mathscr G^{\mu,\nu}_p(u)
\le 
\|Du\|_{L^p}^p\left(\mu(\mathbb S^{N-1})+\nu(\R^N\setminus\set*{0})\right)
\end{equation}
for every $u\in {\sob{p}}(\R^N)$.
Moreover, if $(u_k)_{k\in\N}\subset {\sob{p}}(\R^N)$ is such that $u_k\to u$ in $L^p(\R^N)$ as $k\to\infty$ for some $u\in {\sob{p}}(\R^N)$, then 
\begin{equation}
\label{eq:G_lsc}
\liminf_{k\to\infty}
\mathscr G^{\mu,\nu}_p(u_k)
\ge 
\mathscr G^{\mu,\nu}_p(u).
\end{equation} 
\end{lemma}

The following results build upon the ideas contained in the proof of~\cite{DDP24}*{Th.~1.1}.

\begin{lemma}
\label{res:deltoide}
If $\lambda,\mu\in\mathscr M^+(\mathbb S^{N-1})$, $\nu\in\mathscr M^+(\R^N)$ and $\alpha\in[0,\infty)$ are such that $\mathscr G^{\mu,\nu}_p(u)=\mathscr G^{\lambda,\alpha\delta_0}_p(u)$ for every $u\in\operatorname{Lip}_c(\R^N)$, then $\nu=\beta\delta_0$ for some $\beta\in[0,\infty)$.
\end{lemma}

\begin{proof}
We let  $u\in\operatorname{Lip}_c(\R^N)$ and we define $u_\e=\e^{1-\frac Np}\,u(\cdot/\e)$ for every $\e>0$.
We note that $u_\e\in\operatorname{Lip}_c(\R^N)$ for every $\e>0$.
Moreover, we have that
\begin{equation*}
\|\sigma\cdot Du_\e\|_{L^p}^p
=
\left\|\e^{-\frac Np}\,\sigma\cdot Du(\cdot/\e)\right\|_{L^p}^p
=
\|\sigma\cdot Du\|_{L^p}^p
\end{equation*}
for every $\sigma\in\mathbb S^{N-1}$ and, similarly,
\begin{equation*}
\|u_\e(\,\cdot+z)-u_\e\|_{L^p}^p
=
\e^N
\left\|\e^{1-\frac Np}\left(u\left(\,\cdot+\tfrac z\e\right)-u\right)\right\|_{L^p}^p
=
\e^p\left\|u\left(\,\cdot+\tfrac z\e\right)-u\right\|_{L^p}^p
\end{equation*}
for every $z\in\R^N$.
Thus, the equality $\mathscr G^{\mu,\nu}_p(u_\e)=\mathscr G^{\lambda,\alpha\delta_0}_p(u_\e)$ equivalently rewrites as
\begin{equation}
\label{eq:pallina}
\int_{\mathbb S^{N-1}}\|\sigma\cdot Du\|_{L^p}^p\di\mu(\sigma)
+
\e^p\int_{\R^N\setminus\set*{0}}\frac{\left\|u\left(\,\cdot+\tfrac z\e\right)-u\right\|_{L^p}^p}{|z|^p}
\di\nu(z)
=
\int_{\mathbb S^{N-1}}\|\sigma\cdot Du\|_{L^p}^p\di\lambda(\sigma)
\end{equation}  
for every $\e>0$. 
We claim that
\begin{equation}
\label{eq:ace}
\lim_{\e\to0^+}
\e^p\int_{\R^N\setminus\set*{0}}
\frac{\left\|u\left(\,\cdot+\tfrac z\e\right)-u\right\|_{L^p}^p}{|z|^p}
\di\nu(z)
=0.
\end{equation} 
Indeed, by \cref{res:ftc}\ref{item:ftc1}, we can estimate
\begin{equation}
\label{eq:ping}
\e^p\,\frac{\left\|u\left(\,\cdot+\tfrac z\e\right)-u\right\|_{L^p}^p}{|z|^p}
\le\|Du\|_{L^p}^p
\end{equation}
for every $z\in\R^N\setminus\set*{0}$ and $\e>0$.
Moreover, we have that
\begin{equation}
\label{eq:pong}
\limsup_{\e\to0^+}
\e^p\,\frac{\left\|u\left(\,\cdot+\tfrac z\e\right)-u\right\|_{L^p}^p}{|z|^p}
\le
\limsup_{\e\to0^+} 
\e^p
\,
\frac{2^p\|u\|_{L^p}^p}{|z|^p}
=0
\end{equation}
for every $z\in\R^N\setminus\set*{0}$.
Owing to~\eqref{eq:ping} and~\eqref{eq:pong},   the claimed~\eqref{eq:ace} follows by the Dominated Convergence Theorem applied with respect to~$\nu$.
Therefore, passing to the limit in~\eqref{eq:pallina} as $\e\to0^+$ and exploiting~\eqref{eq:ace}, we conclude that 
\begin{equation*}
\int_{\mathbb S^{N-1}}\|\sigma\cdot Du\|_{L^p}^p\di\mu(\sigma)
=
\int_{\mathbb S^{N-1}}\|\sigma\cdot Du\|_{L^p}^p\di\lambda(\sigma)
\end{equation*}  
whenever $u\in\operatorname{Lip}_c(\R^N)$.
By our initial assumption, this means that 
\begin{equation*}
\int_{\R^N\setminus\set*{0}}\frac{\|u(\,\cdot+z)-u\|_{L^p}^p}{|z|^p}
\di\nu(z)
=
0
\end{equation*}
for every $u\in\operatorname{Lip}_c(\R^N)$. 
In particular, if $u\in\operatorname{Lip}_c(\R^N)$ is such that $\operatorname{supp} u\subset B_{\e/2}$ for some $\e>0$, then 
\begin{equation*}
0
\ge 
\int_{B_\e^c}\frac{\|u(\,\cdot+z)-u\|_{L^p}^p}{|z|^p}
\di\nu(z)
= 
2\|u\|_{L^p}^p
\int_{B_\e^c}
\frac{\di\nu(z)}{|z|^p},
\end{equation*}
from which we deduce that $\nu(B_\e^c)=0$ whenever $\e>0$.
Thus $\nu(\R^N\setminus\set*{0})=0$, from which we get that $\nu=\beta\delta_0$ for some $\beta\in[0,\infty)$, concluding the proof.
\end{proof}

\section{Proof of \texorpdfstring{\cref{resi:bbm,resi:maxrank}}{Theorems 1.1 and 1.2}}
\label{sec:bbm_sharp}

Throughout this section, we let $p\in[1,\infty)$ and $(\rho_t)_{t\in I}\subset L^1_{\rm loc}(\R^N)$ be such that $\rho_t\ge0$ for every $t\in I$.
To prove \cref{resi:bbm,resi:maxrank}, we need the following preliminary result, which improves and generalizes~\cite{DDP24}*{Th.~1.2}.

\begin{theorem}
\label{res:munu_p}
Let $J\subset I$ be such that $0\in\overline J$. 
If any of the properties in \cref{res:diavolacci} hold, then there exists an infinitesimal sequence $(t_k)_{k\in\N}\subset J$ and two measures $\mu\in\mathscr M^+(\mathbb S^{N-1})$ and $\nu\in\mathscr M^+(\R^N)$ such that the following hold:

\begin{enumerate}[label=(\roman*),itemsep=1ex,topsep=1ex]

\item
\label{item:munu_p_nu_limit}
$\rho_{t_k}\,\mathscr L^N\weakstarto\nu$ in $\mathscr M_{\rm loc}(\R^N)$ as $k\to\infty$;

\item 
\label{item:munu_p_limsup}
if $u\in {\sob{p}}(\R^N)$, then
$\displaystyle\limsup_{k\to\infty}
\mathscr F_{t_k,p}(u)
\le
\mathscr G^{\mu,\nu}_p(u)$;

\item
\label{item:munu_p_Gamma-liminf}
if $(u_k)_{k\in\N}\subset L^p(\R^N)$ is such that $u_k\to u$ in $L^p(\R^N)$ as $k\to\infty$ for some $u\in {\sob{p}}(\R^N)$, then
$\displaystyle
\liminf_{k\to\infty}
\mathscr F_{t_k,p}(u_k)
\ge 
\mathscr G^{\mu,\nu}_p(u)$.

\end{enumerate}

As a consequence, $(\mathscr F_{t_k,p})_{k\in\N}$ converges to $\mathscr G_p^{\mu,\nu}$ on ${\sob{p}}(\R^N)$ as $k\to\infty$ pointwise and in the $\Gamma$-sense with respect to the $L^p$ topology.
In addition, if $(\rho_t)_{t\in J}$ has maximal rank, then the family $(\mathscr F_{t_k,p})_{k\in\N}$ is coercive on ${\sob{p}}(\R^N)$.
\end{theorem}

\begin{proof}
For every measurable set $A\subset\R^N$, $u\in {\sob{p}}(\R^N)$ and $t\in I$, we set
\begin{equation*}
\mathscr F_{t,p}(u;A)
=
\int_A\frac{\|u(\,\cdot+z)-u\|_{L^p}^p}{|z|^p}\,\rho_t(z)\di z.
\end{equation*}
Note that $\mathscr F_{t,p}(u;\R^N)=\mathscr F_{t,p}(u)$ for every $t\in I$ and $u\in {\sob{p}}(\R^N)$.
Owing to \cref{res:diavolacci}, there exists $M>0$ such that  
\begin{equation}
\label{eq:rho_M}
\limsup_{t\to0^+}
\int_{B_R}\rho_t(z)\di z+
\limsup_{t\to0^+}
R^p\int_{B_R^c}\frac{\rho_t(z)}{|z|^p}\di z
\le 
M
\end{equation}
for every $R>0$.
By~\eqref{eq:rho_M}, we can apply \cref{res:munu_limit} and find a countable set $I_0\subset I$, two infinitesimal sequences $(t_k)_{k\in\N}\subset J$ and $(\delta_l)_{l\in\N}\subset I\setminus I_0$, and two measures  $\mu\in\mathscr M^+(\mathbb S^{N-1})$ and $\nu\in\mathscr M^+(\R^N)$ satisfying the  statements~\ref{item:munu_limit1}, \ref{item:munu_limit2} and~\ref{item:munu_limit3} of \cref{res:munu_limit}.
In particular, this immediately yields~\ref{item:munu_p_nu_limit}.
We shall prove~\ref{item:munu_p_limsup} and~\ref{item:munu_p_Gamma-liminf} separately.
In what follows, we let $(\mu_k^l)_{k,l\in\N}\subset\mathscr M^+(\mathbb S^{N-1})$, $\mu_k^l\coloneqq\mu_{t_k}^{\delta_l}$ for every $k,l\in\N$, and $(\mu^l)_{l\in\N}\subset\mathscr M^+(\mathbb S^{N-1})$ be given as in statement~\ref{item:munu_limit3} of \cref{res:munu_limit}.

\vspace{1ex}

\textit{Proof of \ref{item:munu_p_limsup}}.
We begin by observing that
\begin{equation}
\label{eq:split}
\mathscr F_{t,p}(u;\R^N)
=
\mathscr F_{t,p}(u;B_\delta)
+
\mathscr F_{t,p}(u;A_\delta)
+
\mathscr F_{t,p}(u;B_{1/\delta}^c)
\end{equation}
for every $t,\delta\in I$, where $A_\delta=\set*{x\in\R^N : \delta<|x|<\frac1\delta}$ as in \cref{res:munu_limit}.
We now deal with each piece on the right-hand side of~\eqref{eq:split} separately. 
By \cref{res:ftc}\ref{item:ftc1}, we can estimate
\begin{equation*}
\mathscr F_{t_k,p}(u;B_{\delta_l})
\le 
\int_{B_{\delta_l}}\left\|\tfrac{z}{|z|}\cdot Du\right\|_{L^p}^p\di\nu_k(z)
=
\int_{\mathbb S^{N-1}}
\|\sigma\cdot Du\|_{L^p}^p\di\mu_k^l(\sigma)
\end{equation*}
for every $u\in {\sob{p}}(\R^N)$ and $k,l\in\N$ (recall the definition in~\eqref{eq:spherical_meas} in \cref{res:munu_limit}).
Since $\sigma\mapsto\|\sigma\cdot Du\|_{L^p}^p\in C(\mathbb S^{N-1})$ for every $u\in {\sob{p}}(\R^N)$, by \cref{res:munu_limit}\ref{item:munu_limit3} we get that 
\begin{equation}
\begin{split}
\label{eq:bim}
\lim_{l\to\infty}
&
\limsup_{k\to\infty}
\mathscr F_{t_k,p}(u;B_{\delta_l})
\le 
\lim_{l\to\infty}
\lim_{k\to\infty}
\int_{\mathbb S^{N-1}}
\|\sigma\cdot Du\|_{L^p}^p\di\mu_k^l(\sigma)
\\
&=
\lim_{l\to\infty}
\int_{\mathbb S^{N-1}}
\|\sigma\cdot Du\|_{L^p}^p\di\mu^l(\sigma)
=
\int_{\mathbb S^{N-1}}
\|\sigma\cdot Du\|_{L^p}^p\di\mu(\sigma).
\end{split}
\end{equation}
Moreover, since 
\begin{equation}
\label{eq:fu}
z\mapsto f_u(z)=\frac{\|u(\,\cdot+z)-u\|_{L^p}^p}{|z|^p}\in C(\R^N\setminus\set{0})
\end{equation}
for every $u\in {\sob{p}}(\R^N)$,
by \cref{res:munu_limit}\ref{item:munu_limit2} and by the Monotone Convergence Theorem, we also have that 
\begin{equation}
\label{eq:bum}
\begin{split}
\lim_{l\to\infty}
\lim_{k\to\infty}
\mathscr F_{t_k,p}(u;A_l)
&=
\lim_{l\to\infty}
\lim_{k\to\infty}
\int_{A_l}
f_u\di\nu_k
=
\lim_{l\to\infty}
\int_{A_l}
f_u\di\nu
=
\int_{\R^N\setminus\set{0}}
f_u\di\nu
\\
&=
\int_{\R^N\setminus\set{0}}
\frac{\|u(\cdot+z)-u\|_{L^p}^p}{|z|^p}\di \nu(z)
\end{split}
\end{equation}
for every $u\in {\sob{p}}(\R^N)$, where $A_l=A_{\delta_l}$ for every $l\in\N$ as in \cref{res:munu_limit}\ref{item:munu_limit2}.
Finally, by exploiting~\eqref{eq:rho_M}, we can estimate
\begin{equation}
\label{eq:bam}
\begin{split}
\lim_{l\to\infty}
\limsup_{k\to\infty}
\mathscr F_{t_k,p}(u;B_{1/\delta_l}^c)
&\le
2^p\,\|u\|_{L^p}^p 
\lim_{l\to\infty}
\limsup_{k\to\infty}
\int_{B_{1/\delta_l}^c}
\frac{\rho_{t_k}(z)}{|z|^p}\di z
\\
&\le 
2^p\,\|u\|_{L^p}^p\,M
\lim_{l\to\infty}
\delta^p_l=0
\end{split}
\end{equation}
for every $u\in {\sob{p}}(\R^N)$. 
By combining~\eqref{eq:bim}, \eqref{eq:bum} and~\eqref{eq:bam} with~\eqref{eq:split}, we get~\ref{item:munu_p_limsup}.

\vspace{1ex}

\textit{Proof of \ref{item:munu_p_Gamma-liminf}}.
Let $(u_k)_{k\in\N}\subset L^p(\R^N)$ and $u\in {\sob{p}}(\R^N)$ be such that $u_k\to u$ in $L^p(\R^N)$ as $k\to\infty$.
Let $(\eta_j)_{j\in\N}\subset C^\infty_c(\R^N)$ be a sequence of mollifiers and set $u_k^j=u_k*\eta_j$ and $u^j=u*\eta_j$ for every $k,j\in\N$.
We observe that $u_k^j,u^j\in {\sob{p}}(\R^N)\cap C^\infty(\R^N)$ with $D^2 u_k^j,D^2u^j\in L^p(\R^N)$ for every $k,j\in\N$.
Moreover, by Young's inequality, we have that
\begin{equation}
\label{eq:chiurlo}
u_k^j\to u^j,
\
Du_k^j\to Du^j
\
\text{and}
\
D^2u_k^j\to D^2u^j
\
\text{in $L^p(\R^N)$ as $k\to\infty$},
\end{equation}
for every $j\in\N$.
In addition, again by Young's inequality, 
\begin{equation*}
\|u_k^j(\,\cdot+z)-u_k^j\|_{L^p}
=
\|(u_k(\,\cdot+z)-u_k)*\eta_j\|_{L^p}
\le 
\|u_k(\,\cdot+z)-u_k\|_{L^p},
\end{equation*}
for every $k,j\in\N$ and $z\in\R^N$,
from which we get that $\mathscr F_{t,p}(u_k^j;A)\le\mathscr F_{t,p}(u_k;A)$ for every measurable set $A\subset\R^N$ and every $k,j\in\N$.
Consequently, we have that 
\begin{equation*}
\liminf_{k\to\infty}
\mathscr F_{t_k,p}(u_k)
\ge 
\liminf_{k\to\infty}
\mathscr F_{t_k,p}(u_k^j)
\end{equation*} 
for every $j\in\N$.
We now claim that 
\begin{equation}
\label{eq:poncio}
\lim_{k\to\infty}
\mathscr F_{t_k,p}(u_k^j)
=
\mathscr G^{\mu,\nu}_p(u^j)
\end{equation}
for every $j\in\N$.
To prove~\eqref{eq:poncio}, we argue as in the proof of~\ref{item:munu_p_limsup} by again relying on~\eqref{eq:split}.
Indeed, as in~\eqref{eq:bam}, observing that $\|u_k^j\|_{L^p}\le\|u_k\|_{L^p}$ for every $k,j\in\N$, we have that 
\begin{equation*}
\lim_{l\to\infty}
\lim_{k\to\infty}
\mathscr F_{t_k,p}(u_k^j;B_{1/\delta_l}^c)
=
2^p\,\|u^j\|_{L^p}^p\,M
\lim_{l\to\infty}
\delta^p_l=0
\end{equation*}
for every $j\in\N$.
Moreover, as in~\eqref{eq:bum}, since $f_{u_k^j}\to f_{u^j}$ locally uniformly on $\R^N\setminus\set*{0}$ as $k\to\infty$ for every $j\in\N$ owing to~\eqref{eq:chiurlo} (for the notation, recall~\eqref{eq:fu}), we also have that 
\begin{equation*}
\begin{split}
\lim_{l\to\infty}
\lim_{k\to\infty}
\mathscr F_{t_k,p}(u_k^j;A_l)
&=
\lim_{l\to\infty}
\lim_{k\to\infty}
\int_{A_l}f_{u_k^j}\di\nu_k
=
\lim_{l\to\infty}
\int_{A_l}f_{u^j}\di\nu
=
\int_{\R^N\setminus\set*{0}}f_{u^j}\di\nu
\\
&=
\int_{\R^N\setminus\set{0}}
\frac{\|u^j(\cdot+z)-u^j\|_{L^p}^p}{|z|^p}\di \nu(z)
\end{split}
\end{equation*}
for every $j\in\N$.
In addition, as in~\eqref{eq:bim}, since the functions $\sigma\mapsto\|\sigma\cdot Du_k^j\|_{L^p}^p$ converge uniformly on $\mathbb S^{N-1}$ to $\sigma\mapsto\|\sigma\cdot Du^j\|_{L^p}^p$ as $k\to\infty$ for every $j\in\N$ owing to~\eqref{eq:chiurlo}, we get
\begin{equation*}
\begin{split}
\lim_{l\to\infty}
\lim_{k\to\infty}
\int_{B_{\delta_l}}
&\left\|\tfrac{z}{|z|}\cdot Du_k^j\right\|_{L^p}^p\di\nu_k(z)
=
\lim_{l\to\infty}
\lim_{k\to\infty}
\int_{\mathbb S^{N-1}}
\|\sigma\cdot Du_k^j\|_{L^p}^p\di\mu_k^l(\sigma)
\\
&=
\lim_{l\to\infty}
\int_{\mathbb S^{N-1}}
\|\sigma\cdot Du^j\|_{L^p}^p\di\mu^l(\sigma)
=
\int_{\mathbb S^{N-1}}
\|\sigma\cdot Du^j\|_{L^p}^p\di\mu(\sigma)
\end{split}
\end{equation*}
for every $j\in\N$ (recall the definition in~\eqref{eq:spherical_meas} in \cref{res:munu_limit}).
Therefore, in view of~\eqref{eq:split}, the claim in~\eqref{eq:poncio} follows if we prove that
\begin{equation*}
\lim_{l\to\infty}
\lim_{k\to\infty}
\mathscr F_{t_k,p}(u_k^j;B_{\delta_l})
=
\int_{\mathbb S^{N-1}}
\|\sigma\cdot Du^j\|_{L^p}^p\di\mu(\sigma);
\end{equation*} 
that is, equivalently, we just have to show that
\begin{equation}
\label{eq:macho}
\lim_{l\to\infty}
\lim_{k\to\infty}
\left|
\mathscr F_{t_k,p}(u_k^j;B_{\delta_l})
-
\int_{B_{\delta_l}}\left\|\tfrac{z}{|z|}\cdot Du_k^j\right\|_{L^p}^p\di\nu_k(z)
\right|
=
0
\end{equation}
for every $j\in\N$.
To this aim, we start by noticing that, by \cref{res:ftc}\ref{item:ftc2},  
\begin{equation*}
\left|
\frac{\|u_k^j(\,\cdot+z)-u_k^j\|_{L^p}}{|z|}
-
\left\|\tfrac{z}{|z|}\cdot Du_k^j\right\|_{L^p}
\right|
\le 
\frac{|z|}{2}\,\|D^2u_k^j\|_{L^p}
\end{equation*}
for every $z\in\R^N\setminus\set*{0}$.
Hence, since $|a^p-b^p|\le p\max\set*{a,b}^{p-1}|a-b|$, for all $a,b\ge0$, and 
\begin{equation*}
\max\set*{
\frac{\|u_k^j(\,\cdot+z)-u_k^j\|_{L^p}}{|z|},
\left\|\tfrac{z}{|z|}\cdot Du_k^j\right\|_{L^p}
}
\le 
\|Du_k^j\|_{L^p}
\end{equation*}
for every $k,j\in\N$ and $z\in\R^N\setminus\set*{0}$ by \cref{res:ftc}\ref{item:ftc1}, we can estimate
\begin{equation*}
\begin{split}
\left|
\frac{\|u_k^j(\,\cdot+z)-u_k^j\|_{L^p}^p}{|z|^p}
-
\left\|\tfrac{z}{|z|}\cdot Du_k^j\right\|_{L^p}^p
\right|
&\le
p\,
\|Du_k^j\|_{L^p}^{p-1}
\,
\left|
\frac{\|u_k^j(\,\cdot+z)-u_k^j\|_{L^p}}{|z|}
-
\left\|\tfrac{z}{|z|}\cdot Du_k^j\right\|_{L^p}
\right|
\\
&\le 
\frac{p|z|}{2}
\,
\|Du_k^j\|_{L^p}^{p-1}
\,
\|D^2u_k^j\|_{L^p},
\end{split}
\end{equation*} 
for every $i,j\in\N$ and $z\in\R^N\setminus\set*{0}$.
Consequently, we get that
\begin{equation*}
\left|
\mathscr F_{t_k,p}(u_k^j;B_{\delta_l})
-
\int_{B_{\delta_l}}\left\|\tfrac{z}{|z|}\cdot Du_k^j\right\|_{L^p}^p\di\nu_k(z)
\right|
\le 
\frac{p}{2}
\,
\|Du_k^j\|_{L^p}^{p-1}
\,
\|D^2u_k^j\|_{L^p}
\,
\delta_l\,\nu_k(B_{\delta_l})
\end{equation*}
for every $k,j,l\in\N$, from which, owing to~\eqref{eq:rho_M} and~\eqref{eq:chiurlo}, the claimed~\eqref{eq:macho} follows.
We thus completed the proof of~\eqref{eq:poncio} and so, by the lower semicontinuity of $\mathscr G^{\mu,\nu}_p$ with respect to the $L^p$ convergence of functions in ${\sob{p}}(\R^N)$ (recall~\eqref{eq:G_lsc}), 
\begin{equation*}
\liminf_{k\to\infty}
\mathscr F_{t_k,p}(u_k)
\ge 
\liminf_{j\to\infty}
\mathscr G^{\mu,\nu}_p(u^j)
=
\mathscr G^{\mu,\nu}_p(u),
\end{equation*}
concluding the proof of~\ref{item:munu_p_Gamma-liminf}.

\vspace{1ex}
With~\ref{item:munu_p_limsup} and~\ref{item:munu_p_Gamma-liminf} in force, the convergence of $(\mathscr F_{t_k,p})_{k\in\N}$ to $\mathscr G_p^{\mu,\nu}$ on ${\sob{p}}(\R^N)$ as $k\to\infty$ in the pointwise and $\Gamma$-sense with respect to the $L^p$ topology follows by \cref{def:gamma_conv}.
We are thus left to prove that, if $(\rho_t)_{t\in J}$ has maximal rank, then $(\mathscr F_{t_k,p})_{k\in\N}$ is coercive. 
Indeed, let $(u_k)_{k\in\N}\subset L^p(\R^N)$ be such that $u_k\to u$ in $L^p(\R^N)$ as $k\to\infty$ for some $u\in L^p(\R^N)$ and   $C=\liminf\limits_{k\to\infty}
\mathscr F_{t_k,p}(u_k)\in[0,\infty)$.
Let us define $u_k^j=u_k*\eta_j$ and $u^j=u*\eta_j$ for every $k,j\in\N$ as in the proof of~\ref{item:munu_p_Gamma-liminf}.
We observe that $u_k^j,u\in {\sob{p}}(\R^N)$ with $\mathscr F_{t,p}(u_k^j)\le\mathscr F_{t,p}(u_k)$ for every $k,j\in\N$ and that $u_k^j\to u^j$ in $L^p(\R^N)$ as $k\to\infty$ for every $j\in\N$.
Therefore, thanks to~\ref{item:munu_p_Gamma-liminf}, we get that 
\begin{equation*}
C
\ge 
\liminf_{k\to\infty}
\mathscr F_{t_k,p}(u_k^j)
\ge 
\mathscr G_p^{\mu,\nu}(u^j)
\end{equation*} 
for every $j\in\N$.
We now note that
\begin{equation*}
\mathscr G_p^{\mu,\nu}(u^j)
\ge 
\int_{\mathbb S^{N-1}}\|\sigma\cdot Du^j\|_{L^p}^p\di\mu(\sigma)
=
\int_{\R^N}
\int_{\mathbb S^{N-1}}
|\sigma\cdot Du^j(x)|^p
\di\mu(\sigma)
\di x.
\end{equation*}
By Jensen's inequality, and recalling the notation introduced in~\eqref{eq:wei}, we have that 
\begin{equation*}
\begin{split}
\int_{\mathbb S^{N-1}}
|\sigma\cdot Du^j(x)|^p
\di\mu(\sigma)
&\ge 
\mu(\mathbb S^{N-1})^{1-p}
\left(\int_{\mathbb S^{N-1}}
|\sigma\cdot Du^j(x)|
\di\mu(\sigma)
\right)^p
\\
&=
\mu(\mathbb S^{N-1})^{1-p}
\,
\Theta_\mu(Du^j(x))^p,
\end{split}
\end{equation*} 
so that $\sup_{j\in\N}\|\Theta_\mu(Du^j)\|_{L^p}<\infty$.
Now, since $(\rho_t)_{t\in J}$ has maximal rank, by \cref{res:munu_limit} we get that $\operatorname{span}(\operatorname{supp}\mu)=\R^N$ and thus, by \cref{res:wei}, we infer that 
\begin{equation*}
\alpha\coloneqq\min_{v\in\mathbb S^{N-1}}\Theta_\mu(v)\in(0,\infty).
\end{equation*}
As a consequence, we hence deduce that 
\begin{equation*}
\infty
>
\sup_{j\in\N}\|\Theta_\mu(Du^j)\|_{L^p}
\ge 
\alpha 
\sup_{j\in\N}\|Du^j\|_{L^p},
\end{equation*}
proving that $(u^j)_{j\in\N}$ is a bounded sequence in $\sob{p}(\R^N)$.
Since $u^j\to u$ in $L^p(\R^N)$ as $j\to\infty$, we get that $u\in {\sob{p}}(\R^N)$, concluding the proof.
\end{proof}

\begin{proof}[Proof of \cref{resi:bbm,resi:maxrank}]
We begin with the proof of \cref{resi:bbm}, showing the equivalence between~\ref{itemi:bbm_kernels} and~\ref{itemi:bbm_limits} by dealing with the two implications separately.

\vspace{1ex}

\textit{Proof of \ref{itemi:bbm_kernels}$\implies$\ref{itemi:bbm_limits}}.
If~\ref{itemi:bbm_kernels} holds, then by \cref{res:munu_p} we find two measures $\mu\in\mathscr M^+(\mathbb S^{N-1})$ and $\nu\in\mathscr M^+(\R^N)$ such that $(\mathscr F_{t_k,p})_{k\in\N}$ converges to $\mathscr G_p^{\mu,\nu}$ on ${\sob{p}}(\R^N)$ as $k\to\infty$ pointwise and in the $\Gamma$-sense with respect to the $L^p$ topology.
Actually, \cref{res:munu_p}\ref{item:munu_p_nu_limit} yields that $\nu=\alpha\delta_0$ for some $\alpha\in[0,\infty)$, so that $\mathscr G_p^{\mu,\alpha\delta_0}=\mathscr G_p^{\mu,0}=\mathscr D_p^\mu$ on ${\sob{p}}(\R^N)$, proving~\ref{itemi:bbm_limits}.

\vspace{1ex}

\textit{Proof of \ref{itemi:bbm_limits}$\implies$\ref{itemi:bbm_kernels}}.
If~\ref{itemi:bbm_limits} holds, then~\eqref{eq:G_torta} yields that 
\begin{equation*}
\lim_{k\to\infty}
\mathscr F_{t_k,p}(u)
=
\mathscr D_p^\mu(u)=
\mathscr G^{\mu,0}_p(u)
\le 
\mu(\mathbb S^{N-1})
\,
\|Du\|_{L^p}^p
\end{equation*}
for every $u\in {\sob{p}}(\R^N)$.
We can thus apply \cref{res:pulce} and then \cref{res:diavolacci} to get the first part of~\ref{itemi:bbm_kernels}.
This, in turn, allows us to apply \cref{res:munu_p} and find two measures $\lambda\in\mathscr M^+(\mathbb S^{N-1})$ and $\nu\in\mathscr M^+(\R^N)$ such that $\rho_{t_k}\,\mathscr L^N\weakstarto\nu$ in $\mathscr M_{\rm loc}(\R^N)$ as $k\to\infty$ and, moreover, $(\mathscr F_{t_k,p})_{k\in\N}$ converges to $\mathscr G_p^{\lambda,\nu}$ on ${\sob{p}}(\R^N)$ as $k\to\infty$ pointwise and in the $\Gamma$-sense with respect to the $L^p$ topology.
Because of~\ref{itemi:bbm_limits}, this means that $\mathscr G_p^{\lambda,\nu}=\mathscr D_p^\mu=\mathscr G_p^{\mu,0}$ on ${\sob{p}}(\R^N)$ and thus, by \cref{res:deltoide}, we conclude that $\nu=\alpha\delta_0$ for some $\alpha\in[0,\infty)$, proving the second part of~\ref{itemi:bbm_kernels}.   

\vspace{1ex}
To prove \cref{resi:maxrank}, so further assuming that $(\rho_{t_k})_{k\in\N}$ is of maximal rank as in \cref{def:maxrank}, it is enough to observe that, if~\ref{itemi:bbm_kernels} or~\ref{itemi:bbm_limits} holds, then \cref{res:munu_p} yields that $(\mathscr F_{t_k,p})_{k\in\N}$ is coercive on ${\sob{p}}(\R^N)$.
\end{proof}

\begin{remark}
\label{rem:tk_uguale}
Given $m\in\N$ such that $m\le N$, write $\R^N=\R^m\times\R^{N-m}$ and $x=(x',x'')$ accordingly, and let $B_r^m=\set*{x'\in\R^m: |x'|<r}$ and $B_r^{N-m}=\set*{x''\in\R^{N-m}: |x''|<r}$ for all $r>0$.
The family $(\rho_t^{(m,1)})_{t\in I}\subset L^1(\R^N)$ defined as 
\begin{equation*}
\rho_t^{(m,1)}\coloneqq\frac{\chi_{B_t^m\times B_{t^2}^{N-m}}}{\mathscr L^m(B_t^m)\cdot\mathscr L^{N-m}(B_{t^2}^{N-m})},
\quad
t\in I,
\end{equation*}
is such that $\|\rho^{(m,1)}_t\|_{L^1}=1$ for all $t\in I$ (and thus $(\rho^{(m,1)}_t)_{t\in I}$ obviously satisfies~\eqref{eqi:bbm_suff} as $t\to0^+$) and $\nu_t^{(m,1)}\coloneqq\rho_t^{(m,1)}\mathscr L^N\weakstarto\delta_0$ in $\mathscr M_{\rm loc}(\R^N)$ as $t\to0^+$. 
Therefore, the family $(\rho_t^{(m,1)})_{t\in I}$ satisfies part \ref{itemi:bbm_kernels} of \cref{resi:bbm} along \emph{any} infinitesimal sequence $(t_k)_{k\in\N}\subset I$.  
By~\cite{P04a}*{Ex.~3}, part \ref{itemi:bbm_limits} of \cref{resi:bbm} holds with
\begin{equation*}
\mu^{(m,1)}=c_{m,1}\,\mathscr H^{m-1}\mres\mathbb S^{m-1,1},
\end{equation*}
where 
\begin{equation*}
\mathbb S^{m-1,1}\coloneqq\set*{x\in\R^N : |x'|=1},
\quad
c_{m,1}\coloneqq\mathscr H^{m-1}(\mathbb S^{m-1,1})^{-1},
\end{equation*}
along \emph{any} infinitesimal sequence $(t_k)_{k \in\N}$.
In particular, the limit Dirichlet energy is 
\begin{equation*}
\mathscr D^{(m,1)}_p(u)
=
c_m\int_{\mathbb S^{m-1,1}}\|\sigma\cdot Du\|_{L^p}^p\di\mathscr H^{m-1}(\sigma),
\quad
u\in\sob p(\R^N),
\end{equation*}
for any $p\in[1,\infty)$, see~\cite{P04a}*{Cor.~2} and the comments below it.
\textit{Mutatis mutandis}, also the family $(\rho_t^{(m,2)})_{t\in I}\subset L^1(\R^N)$ defined as
\begin{equation*}
\rho_t^{(m,2)}\coloneqq\frac{\chi_{B_{t^2}^m\times B_t^{N-m}}}{\mathscr L^m(B_{t^2}^m)\cdot\mathscr L^{N-m}(B_t^{N-m})},
\quad
t\in I,
\end{equation*}
satisfies part \ref{itemi:bbm_kernels} of \cref{resi:bbm} along \emph{any} infinitesimal sequence $(t_k)_{k\in\N}\subset I$, with $\nu_t^{(m,2)}\coloneqq\rho_t^{(m,2)}\mathscr L^N\weakstarto\delta_0$ in $\mathscr M_{\rm loc}(\R^N)$ as $t\to0^+$ and limit Dirichlet energy 
\begin{equation*}
\mathscr D^{(N-m,2)}_p(u)
=
c_{N-m,2}\int_{\mathbb S^{N-m,2}}\|\sigma\cdot Du\|_{L^p}^p\di\mathscr H^{N-m-1}(\sigma),
\quad
u\in\sob p(\R^N),
\end{equation*}
for any $p\in[1,\infty)$, where now
\begin{equation*}
\mathbb S^{N-m-1,2}=\set*{x\in\R^N : |x''|=1},
\quad
c_{N-m,2}=\mathscr H^{N-m-1}(\mathbb S^{N-m-1,2})^{-1}.
\end{equation*} 
Now consider the family $(\rho_t)_{t\in I}$ defined as
\begin{equation*}
\rho_t
=
\begin{cases}
\rho_t^{(1,1)} 
& \text{for $t=\frac1k$ with $k$ odd},
\\[1ex]
\rho_t^{(N-1,2)}
& \text{for $t=\frac1k$ with $k$ even},
\\[1ex]
0 & \text{otherwise}.
\end{cases}
\end{equation*}
From the observations made above we have that $(\rho_t)_{t\in I}$ satisfies part~\ref{itemi:bbm_kernels} in \cref{resi:bbm} along the sequence $t_k=\frac1k$, $k\in\N$, and moreover, being obviously $c_{1,1}=c_{N-1,2}$,
\begin{equation*}
\lim_{\substack{k\to\infty\\k\,\text{odd}}}
\mathscr F_{t_k,p}(u)
=
\|\mathrm e_1\cdot Du\|_{L^p}^p
\quad
\text{and}
\quad
\lim_{\substack{k\to\infty\\k\,\text{even}}}
\mathscr F_{t_k,p}(u)
=
\|\mathrm e_N\cdot Du\|_{L^p}^p,
\end{equation*}
for every $u\in\sob p(\R^N)$ and  $p\in[1,\infty)$.
This means that, for part \ref{itemi:bbm_limits} of \cref{resi:bbm} to hold, one must pass to a subsequence of $(t_k)_{k\in\N}$ in general.
\end{remark}

\section{Proof of \texorpdfstring{\cref{resi:compactness}}{Theorem 1.3}}

\label{sec:special}

In this section we specialize our analysis to a particular class of families of kernels.

\subsection{Special kernels}
\label{subsec:special_kernels}

We let $p\in[1,\infty)$ and $K\colon\R^N\to[0,\infty]$ be a measurable function such that $K\not\equiv0$ and $|\cdot|^p\,K\in L^1_{\rm loc}(\R^N)$.
We set 
\begin{equation}
\label{eq:moment_Kp}
m_{K,p}(R)
=
\int_{B_R}|x|^p\,K(x)\di x\in[0,\infty)
\end{equation}
for all $R>0$.
We note that, since $K\not\equiv0$, there exists $R_0>0$ such that $m_{K,p}(R)>0$ for all $R\ge R_0$.
We let $\beta\colon I\to(0,\infty)$ be a Borel function such that 
\begin{equation}
\label{eq:beta_big}
\lim_{t\to0^+}\beta(t)=\infty,
\end{equation}
and we set 
\begin{equation}
\label{eq:phi_Kp}
\phi_{K,\beta,p}(t)
=
\frac{m_{K,p}(\beta(t))}{\beta(t)^p}
\end{equation}
for all $t>0$.
Owing to~\eqref{eq:beta_big} and the previous definition of $R_0>0$, we can find $t_0\in(0,1)$ such that 
$\phi_{K,\beta,p}(t)>0$ for all $t\in(0,t_0)$.   
Finally, we let $(K_t)_{t\in I}$ be given by
\begin{equation*}
K_t(x)
=
\beta(t)^N
K(\beta(t)x)
\end{equation*}
for each $t>0$ and $x\in\R^N$.
Recalling the previous definition of $t_0\in(0,1)$, we can hence consider the following special family of kernels $(\rho_t)_{t\in I_0}\subset L^1_{\rm loc}(\R^N)$, $I_0=(0,t_0)$, depending on $p$, $K$ and $\beta$, given by
\begin{equation}
\label{eq:rho_Kp}
\rho_t(x)
=
\frac{|x|^p K_t(x)}{\phi_{K,\beta,p}(t)}
\end{equation}
for $t\in I_0$ and $x\in\R^N$.
We observe that, for the family~\eqref{eq:rho_Kp}, the functional $\mathscr F_{t,p}$ in~\eqref{eq:F_energy} can be rewritten as
\begin{equation}
\label{eq:F_Kp}
\mathscr F_{t,p}^{K,\beta}(u)
=
\frac{1}{\phi_{K,\beta,p}(t)}
\int_{\R^N}\int_{\R^N}
|u(x)-u(y)|^p
\,
K_t(x-y)\di x\di y
\end{equation}
for each $t\in I_0$.
Without loss of generality, for simplicity we will assume that $t_0=1$ and thus $I_0=I$ as usual and, unless required for better clarity, we will omit the dependence on~$K$ and~$\beta$ in the quantities of interest to keep the notation short. 

\subsection{Convergence to local energies}

We now apply \cref{resi:bbm} to the special family of kernels given by~\eqref{eq:rho_Kp}.
To this aim, we state the following result, which rephrases \cref{res:munu_limit} for the special family in~\eqref{eq:rho_Kp}.
A similar result has been discussed in~\cite{P04a}*{Ex.~2}.

\begin{proposition}
\label{res:asym_K}
Let $p\in[1,\infty)$, $K\not\equiv0$ and $\beta\colon I\to(0,\infty)$ be as above.
The measures $(\mu_t^\delta)_{t,\delta\in I}$ in~\eqref{eq:spherical_meas} in \cref{res:munu_limit} corresponding to the family $(\rho_t)_{t\in I}$ in~\eqref{eq:rho_Kp} are given by
\begin{equation*}
\mu_t^\delta(E)
=
\int_{E}\left(\int_0^{\beta(t)\delta} r^{N+p-1}K(\sigma r)\di r\right)\frac{\di\mathscr H^{N-1}(\sigma)}{m_{K,p}(\beta(t))}
\end{equation*}
for every $\mathscr H^{N-1}$-measurable set $E\subset\mathbb S^{N-1}$.
Moreover, if
\begin{equation}
\label{eq:supercriticality}
|\cdot|^p\,K\in L^1(\R^N),
\end{equation}
then the following hold:
\begin{enumerate}[label=(\roman*),itemsep=1ex,topsep=1ex]

\item
\label{item:asym_K_mu}
$\mu_t^\delta\weakstarto\theta_{K,p}\,\mathscr H^{N-1}$ in $\mathscr M(\mathbb S^{N-1})$ as $t\to0^+$ for $\delta>0$, where $\theta_{K,p}\colon\mathbb S^{N-1}\to[0,\infty]$ is given by
\begin{equation}
\label{eq:theta_Kp}
\theta_{K,p}
(\sigma)
=
\frac{\displaystyle\int_0^\infty r^{N+p-1}\,K(\sigma r)\di r}{\||\cdot|^p\,K\|_{L^1}}
\quad
\text{for $\mathscr H^{N-1}$-a.e.\ $\sigma\in\mathbb S^{N-1}$;}
\end{equation}

\item
\label{item:asym_K_nu}
the measures $(\nu_t)_{t\in I}$, defined as $\nu_t=\rho_t\,\mathscr L^N$ for every $t\in I$, where the family $(\rho_t)_{t\in I}$ is as in~\eqref{eq:rho_Kp}, satisfy $\nu_t\weakstarto\delta_0$ in $\mathscr M_{\rm loc}(\R^N)$ as $t\to0^+$;

\item
\label{item:asym_K_radial}
if $K$ is radially symmetric, then the family $(\rho_t)_{t\in I}$ in~\eqref{eq:rho_Kp} has maximal rank and $\theta_{K,p}(\sigma)=\frac1{N\omega_N}$
for every $\mathscr H^{N-1}$-a.e.\ $\sigma\in\mathbb S^{N-1}$, so that 
\begin{equation}
\label{eq:asym_K_radial}
\int_{\mathbb S^{N-1}}
\|\sigma\cdot Du\|^p_{L^p}\,\theta_{K,p}(\sigma)\di\mathscr H^{N-1}(\sigma)
=
\frac{2}{N}
\,
\frac{\Gamma\left(\frac{N+1}{2}\right)\Gamma\left(\frac{p+1}{2}\right)}{\Gamma\left(\frac{N+p}{2}\right)}
\,
\|Du\|_{L^p}^p
\end{equation}
for every $u\in {\sob{p}}(\R^N)$.
 
\end{enumerate}
\end{proposition}

\begin{proof}
Let us first observe that, by~\eqref{eq:beta_big} and~\eqref{eq:supercriticality}, we have 
\begin{equation}
\label{eq:m_betaKp_L1}
\lim_{t\to0^+}
m_{K,p}(\beta(t))
=
\||\cdot|^p\,K\|_{L^1}\in(0,\infty).
\end{equation} 
We can now briefly prove each statement separately.

\vspace{1ex}

\textit{Proof of \ref{item:asym_K_mu}}.
Given $\delta>0$ and $\varphi\in C(\mathbb S^{N-1})$, we can compute
\begin{equation*}
\begin{split}
\lim_{t\to0^+}
\int_{\mathbb S^{N-1}}\varphi(\sigma)\di\mu_t^\delta(\sigma)
&=
\lim_{t\to0^+}
\frac{1}{m_{K,p}(\beta(t))}
\int_{\mathbb S^{N-1}}\varphi(\sigma)
\int_0^{\beta(t)\delta}r^{N+p-1}K(\sigma r)\di r\di\mathscr H^{N-1}(\sigma)
\\
&
=
\frac{1}{\||\cdot|^p\,K\|_{L^1}}
\int_{\mathbb S^{N-1}}\varphi(\sigma)
\int_0^{\infty}r^{N+p-1}K(\sigma r)
\di r
\di\mathscr H^{N-1}(\sigma)
\\
&=
\int_{\mathbb S^{N-1}}\varphi(\sigma)
\,
\theta_{K,p}(\sigma)
\di\mathscr H^{N-1}(\sigma)
\end{split}
\end{equation*} 
by~\eqref{eq:beta_big}, \eqref{eq:m_betaKp_L1}, the Dominated Convergence Theorem, and Fubini's Theorem.

\vspace{1ex}

\textit{Proof of \ref{item:asym_K_nu}}.
In view of~\eqref{eq:m_betaKp_L1}, we infer that 
\begin{equation*}
\lim_{t\to0^+}
\nu_t(\R^N)
=
\lim_{t\to0^+}
\|\rho_t\|_{L^1}
=
\lim_{t\to0^+}
\frac{\||\cdot|^p\,K\|_{L^1}}{m_{K,p}(\beta(t))}
=
1.
\end{equation*}
Moreover, if $\varphi\in C_c(\R^N)$, then, by changing variables, we have 
\begin{equation}
\label{eq:fango}
\begin{split}
\lim_{t\to0^+}
\int_{\R^N}\varphi\,\di\nu_t
=
\lim_{t\to0^+}
\frac{1}{m_{K,p}(\beta(t))}
\int_{\R^N}\varphi\left(\frac x{\beta(t)}\right)
|x|^p\,K(x)\di x
=
\varphi(0)
\end{split}
\end{equation}
by~\eqref{eq:beta_big}, \eqref{eq:m_betaKp_L1} and the Dominated Convergence Theorem, owing to~\eqref{eq:supercriticality}. 

\vspace{1ex}

\textit{Proof of~\ref{item:asym_K_radial}}.
Formula~\eqref{eq:asym_K_radial} is well known, see~\cites{P04a,GT24} for instance.
Thus, we just focus on the maximal rank property.
We shall prove that \cref{def:maxrank} is satisfied for the canonical basis of~$\R^N$, $\mathrm e_1,\dots,\mathrm e_N\in\mathbb S^{N-1}$, and any $\tau\in I$ sufficiently small.
Indeed, by radial symmetry, we have
\begin{equation*}
\int_{B_\delta\,\cap\,\mathcal C_\tau(\mathrm e_i)}\rho_t(z)\di z
=
\int_{B_\delta\,\cap\,\mathcal C_\tau(\mathrm e_1)}\rho_t(z)\di z
\end{equation*} 
for all $i\in\set*{1,\dots,N}$.
Arguing as in~\eqref{eq:fango}, we can compute
\begin{equation*}
\int_{B_\delta\,\cap\,\mathcal C_\tau(\mathrm e_1)}\rho_t(z)\di z
=
\frac{1}{m_{K,p}(\beta(t))}
\int_{B_{\delta\beta(t)}\,\cap\,\mathcal C_\tau(\mathrm e_1)}
|x|^p\,K(x)\di x.
\end{equation*}
Therefore, by combining the above equalities, we get that
\begin{equation*}
\liminf_{t\to0^+}
\int_{B_\delta\,\cap\,\mathcal C_\tau(\mathrm e_i)}\rho_t(z)\di z
=
\lim_{t\to0^+}
\frac{1}{m_{K,p}(\beta(t))}
\int_{B_{\delta\beta(t)}\,\cap\,\mathcal C_\tau(\mathrm e_1)}
|x|^p\,K(x)\di x
=
c_{K,p}
\end{equation*}
by~\eqref{eq:beta_big}, \eqref{eq:m_betaKp_L1} and the Monotone Convergence Theorem, where 
\begin{equation*}
c_{K,p}
=
\frac{1}{\||\cdot|^p\,K\|_{L^1}}
\int_{\mathcal C_\tau(\mathrm e_1)}
|x|^p\,K(x)\di x\in(0,1)
\end{equation*}
and the validity of \cref{def:maxrank} readily follows, concluding the proof. 
\end{proof}

We are now ready to apply \cref{resi:bbm} to the special family in~\eqref{eq:rho_Kp}.
We remark that \cref{res:asym_Kp} below was already implicitly given in~\cite{P04a}, although the main results of~\cite{P04a} are stated on bounded open subsets of $\R^N$ with Lipschitz boundary.

\begin{theorem}
\label{res:asym_Kp}
Let $p\in[1,\infty)$, $K\not\equiv0$ and $\beta\colon I\to(0,\infty)$ be as above.
If~\eqref{eq:supercriticality} holds, then the limit
\begin{equation*}
\lim_{t\to0^+}
\beta(t)^p
\int_{\R^N}\int_{\R^N}
|u(x)-u(y)|^p
\,
K_t(x-y)\di x\di y
=
\int_{\mathbb S^{N-1}}
\|\sigma\cdot Du\|^p_{L^p}\,\theta_{K,p}(\sigma)\di\mathscr H^{N-1}(\sigma),
\end{equation*}
where $\theta_{K,p}\colon\mathbb S^{N-1}\to[0,\infty]$ is as in~\eqref{eq:theta_Kp},
holds for $u\in {\sob{p}}(\R^N)$ pointwise and in the $\Gamma$-sense with respect to the $L^p$ topology, and the functionals on the right-hand side are coercive.
If, in addition, $K$ is radially symmetric, then the limit 
\begin{equation*}
\lim_{t\to0^+}
\beta(t)^p
\int_{\R^N}\int_{\R^N}
|u(x)-u(y)|^p
\,
K_t(x-y)\di x\di y
=
\frac{2}{N}
\,
\frac{\Gamma\left(\frac{N+1}{2}\right)\Gamma\left(\frac{p+1}{2}\right)}{\Gamma\left(\frac{N+p}{2}\right)\,\||\cdot|^p K\|_{L^1}}
\,
\|Du\|_{L^p}^p
\end{equation*}
holds for $u\in {\sob{p}}(\R^N)$ pointwise and in the $\Gamma$-sense with respect to the $L^p$ topology, and the functionals on the right-hand side are coercive.
\end{theorem}

\begin{proof}
The statement directly follows  by combining \cref{res:munu_p} with the properties collected in \cref{res:asym_K}.
We omit the simple computations.
\end{proof}

\subsection{Compactness}
We now complete the asymptotic analysis of the energies~\eqref{eq:F_Kp} relative to the special family~\eqref{eq:rho_Kp} achieved in \cref{res:asym_Kp} by proving \cref{resi:compactness}.
To this aim, we need to introduce the following terminology.

\begin{definition}[Local precompactenss]
\label{def:locprecomp}
Let $p\in[1,\infty)$.
A set $\mathcal X\subset L^p(\R^N)$ is locally precompact in $L^p(\R^N)$ if $\mathcal X$ is precompact in $L^p(E)$ for every compact set $E\subset\R^N$.
\end{definition}

For the proof of \cref{resi:compactness}, we need the following preliminary result, which generalizes (the proof of)~\cite{BP19}*{Th.~3.5} to every $p\in[1,\infty)$.

\begin{proposition}
\label{res:supcomp_bound}
Let $p\in[1,\infty)$, $K\not\equiv0$ and $\beta\colon I\to(0,\infty)$ be as above and assume~\eqref{eq:supercriticality}. 
If $u\in L^p(\R^N)$ and $t>0$, then there exists $v_t\in {\sob{p}}(\R^N)$ such that 
\begin{equation*}
\|v_t-u\|_{L^p}
\le 
C_{K,p}
\,
\mathscr F^K_{t,p}(u)
\,
\beta(t)^{-p}
\quad
\text{and}
\quad 
\|\nabla v_t\|_{L^p}
\le
C_{K,p} 
\,
\mathscr F^K_{t,p}(u),
\end{equation*}
where $C_{K,p}>0$ depends on $K$ and $p$ only.
\end{proposition}

In the proof of \cref{res:supcomp_bound} we need the following estimate, which revisits the one in~\cite{BP19}*{Lem.~3.4} for every $p\in[1,\infty)$.
We omit its proof, since it follows by a simple application of Tonelli's Theorem.

\begin{lemma}
\label{res:starlone}
Let $p\in[1,\infty)$.
If $G\in L^1(\R^N)$ is a non-negative function, then 
\begin{equation*}
\int_{\R^N}\|u(\cdot+z)-u\|_{L^p}^p
\,
(G* G)(z)
\di z
\le 
2^p
\,
\|G\|_{L^1}
\int_{\R^N}\|u(\cdot+z)-u\|_{L^p}^p
\,
G(z)
\di z
\end{equation*}
for every $u\in L^p(\R^N)$.
\end{lemma}  

\begin{proof}[Proof of \cref{res:supcomp_bound}]
Since $K\not\equiv0$ and $|\cdot|^p\, K\in L^1(\R^N)$ by~\eqref{eq:supercriticality}, we have that the function $G=\min\set*{K,1}$ satisfies $G\in L^1\cap L^\infty(\R^N)\setminus\set*{0}$, with
\begin{equation*}
\|G\|_{L^1}
\le 
|B_1|
+
\int_{B_1^c}|x|^p\,K(x)\di x
\le 
|B_1|
+
\||\cdot|^p K\|_{L^1}.
\end{equation*}
Hence the function $G*G$ is non-negative, continuous, and strictly positive on a non-empty open set in $\R^N$. 
Thus, we can find a non-negative function $\varphi\in \operatorname{Lip}_c(\R^N)\setminus\set*{0}$ such that
\begin{equation}
\label{eq:testar}
\varphi\le G* G
\quad
\text{and}
\quad
|\nabla\varphi|\le G* G.
\end{equation} 
We hence set
\begin{equation*}
G_t(x)=\beta(t)^NG(\beta(t)x),
\quad
\varphi_t=\frac{\beta(t)^N\varphi(\beta(t)x)}{\|\varphi\|_{L^1}}
\end{equation*}
for every $t>0$ and $x\in\R^N$.
We note that $\|G_t\|_{L^1}=\|G\|_{L^1}$, $\|\varphi_t\|_{L^1}=1$ and, moreover, owing to~\eqref{eq:testar},
\begin{equation}
\label{eq:testart}
\varphi_t
\le 
\frac{G_t* G_t}{\|\varphi\|_{L^1}}
\quad
\text{and}
\quad
|\nabla\varphi_t|
\le 
\frac{G_t* G_t}{\|\varphi\|_{L^1}}
\,
\beta(t)
\end{equation}
for every $t>0$.
Finally, given $u\in L^p(\R^N)$, we set $v_t=u*\varphi_t$ for every $t>0$ and we note that $v_t\in {\sob{p}}(\R^N)$ for every $t>0$. 
Owing to Jensen's inequality, \eqref{eq:testart}, \cref{res:starlone} and the definitions in~\eqref{eq:moment_Kp} and~\eqref{eq:phi_Kp}, we can estimate
\begin{equation*}
\begin{split}
\|v_t&-u\|_{L^p}^p
\le 
\int_{\R^N}\|u(\cdot+z)-u\|_{L^p}^p\,\varphi_t(z)\di z
\le
\frac{1}{\|\varphi\|_{L^1}}
\int_{\R^N}\|u(\cdot+z)-u\|_{L^p}^p\,(G_t* G_t)(z)\di z
\\
&\le
\frac{2^p\|G_t\|_{L^1}}{\|\varphi\|_{L^1}}
\int_{\R^N}\|u(\cdot+z)-u\|_{L^p}^p\,G_t(z)\di z
\le 
\frac{2^p\,\|G\|_{L^1}}{\|\varphi\|_{L^1}}
\,
\int_{\R^N}\|u(\cdot+z)-u\|_{L^p}^p\,K_t(z)\di z
\\
&= 
\frac{2^p\,\|G\|_{L^1}}{\|\varphi\|_{L^1}}
\,
\mathscr F^K_{p,t}(u)
\,
\phi_{K,p}(t)
\le 
\frac{2^p\,\|G\|_{L^1}}{\|\varphi\|_{L^1}}
\,
\||\cdot|^p\,K\|_{L^1}
\,
\mathscr F^K_{p,t}(u)
\,
\beta(t)^{-p}.
\end{split}
\end{equation*}
Moreover, since, for all $x\in\R^N$, by the Divergence Theorem, we can write 
\begin{equation*}
\nabla v_t(x)
=
\int_{\R^N}u(x-z)\,\nabla\varphi_t(z)\di z
=
\int_{\R^N}(u(x-z)-u(x))\,\nabla\varphi_t(z)\di z,
\end{equation*}
we similarly get that
\begin{equation*}
\begin{split}
\|\nabla v_t\|_{L^p}^p
&\le
\|\nabla\varphi_t\|_{L^1}^{p-1}
\int_{\R^N}\|u(\cdot+z)-u\|_{L^p}^p\,|\nabla\varphi_t(z)|\di z
\\
&\le
\|\nabla\varphi\|_{L^1}^{p-1}
\,
\frac{\beta(t)^p}{\|\varphi\|_{L^1}^p}
\int_{\R^N}\|u(\cdot+z)-u\|_{L^p}^p\,(G_t* G_t)(z)\di z 
\\
&\le
2^p\,\|G\|_{L^1}
\,
\frac{\|\nabla\varphi\|_{L^1}^{p-1}}{\|\varphi\|_{L^1}^p}
\,
\||\cdot|^p\,K\|_{L^1}
\,
\mathscr F^K_{p,t}(u)\end{split}
\end{equation*}
for every $t>0$, yielding the conclusion.
\end{proof}

\begin{proof}[Proof of \cref{resi:compactness}]
For convenience, we let 
\begin{equation*}
M=\sup_{k\in\N}
\big(
\|u_k\|_{L^p}
+
\mathscr F_{t_k,p}(u_k)
\big)<\infty.
\end{equation*}
By \cref{res:supcomp_bound} we can find $v_k=v_{t_k}\in {\sob{p}}(\R^N)$ such that $\|v_k-u_k\|_{L^p}
\le 
C\beta(t_k)^{-p}$ and $\|\nabla v_k\|_{L^p}
\le C$ for every $k\in\N$,
where $C>0$ depends on $K$, $p$ and $M$ only.
In particular, the sequence $(v_k)_{k\in\N}$ is bounded in ${\sob{p}}(\R^N)$ and thus we can find a subsequence $(v_{k_j})_{j\in\N}$ and $u\in {\sob{p}}(\R^N)$ such that $v_{k_j}\to u$ in $L^p_{\rm loc}(\R^N)$ as $j\to\infty$.
Since $\beta(t_k)\to\infty$ as $k\to\infty$, we also get that $u_k\to u$ in $L^p_{\rm loc}(\R^N)$ as $k\to\infty$.
A similar argument proves that any $L^p_{\rm loc}(\R^N)$ limit of $(u_k)_{k\in\N}$ belongs to ${\sob{p}}(\R^N)$.
\end{proof}

For future convenience, we complete \cref{resi:compactness} by recalling the following result, which is a consequence of~\cite{P04b}*{Ths.~1.2 and~1.3} (we also refer to the discussion in~\cite{P04b}*{Sec.~2}).

\begin{proposition}
\label{res:ponce_log}
Let $p\in[1,\infty)$.
If $(\lambda_k)_{k\in\N}\subset I$ is infinitesimal and $(u_k)_{k\in\N}\subset L^p(\R^N)$ is such that
\begin{equation*}
\sup_{k\in\N}
\left(
\|u_k\|_{L^p}
+
\int_{\R^N}\frac{\|u_k(\cdot+z)-u_k\|_{L^p}^p}{|\log\lambda_k|\,(\lambda_k+|z|)^{N+p}}\di z
\right)
<\infty,
\end{equation*}
then $(u_k)_{k\in\N}$ is locally precompact in $L^p(\R^N)$ and any of its $L^p_{\rm loc}(\R^N)$ limits is in ${\sob{p}}(\R^N)$.
\end{proposition}

\begin{proof}
The results follows from~\cite{P04b}*{Ths.~1.2 and~1.3}.
Indeed, it is enough to consider the non-negative radial kernels $(\rho_k)_{k\in\N}$ defined as 
\begin{equation*}
\rho_k(z)
=
\frac{c_k|z|^p\,\chi_{B_1}(z)}{|\log\lambda_k|\,(\lambda_k+|z|)^{N+p}},
\end{equation*}
for all $z\in\R^N$ and $k\in\N$, where 
\begin{equation*}
c_k
=
\int_{B_1}
\frac{|z|^p}{|\log\lambda_k|\,(\lambda_k+|z|)^{N+p}}\di z
\end{equation*}
is a renormalization constant such that $\displaystyle\int_{\R^N}\rho_k(z)\di z=1$ for all $k\in\N$.
Since 
\begin{equation*}
0
<
\inf_{k\in\N}c_k
\le
\sup_{k\in\N}c_k
<
\infty,
\end{equation*}
the kernels $(\rho_k)_{k\in\N}$ satisfy the properties in~\cite{P04b}*{Eq.~(2)}, and also the additional property in~\cite{P04b}*{Eq.~(8)} for $N=1$.
We omit the plain details.
\end{proof}

\section{Proof of \texorpdfstring{\cref{resi:diretto_lim,resi:diretto_comp}}{Theorems 1.4 and 1.5}}
\label{sec:non-local}

We now pass to the proof of the non-local stability results, \cref{resi:diretto_lim,resi:diretto_comp}.

\begin{proof}[Proof of \cref{resi:diretto_lim} ]
If $u\in W^{\kappa,p}(\R^N)$, then $z\mapsto\|u(\cdot+z)-u\|_{L^p}^p\,\kappa(z)\in L^1(\R^N)$ and thus, by~\eqref{eqi:diretto_c}, we can apply the Dominated Convergence Theorem to get that
\begin{equation*}
\lim_{t\to0^+}
\mathscr F_{t,p}(u)
=
\int_{\R^N}\|u(\cdot+z)-u\|_{L^p}^p\,\kappa(z)\di z
=
[u]_{W^{\kappa,p}}.
\end{equation*}
Moreover, if $(t_k)_{k\in\N}\subset I$ is infinitesimal and $(u_k)_{k\in\N}\subset L^p(\R^N)$ is such that $u_k\to u$ in $L^p(\R^N)$ as $k\to\infty$, then by~\eqref{eqi:diretto_lim} we can apply Fatou's Lemma to get
\begin{equation*}
\begin{split}
\liminf_{k\to\infty}\mathscr F_{t_k,p}(u_k)
&\ge 
\int_{\R^N}\lim_{k\to\infty}
\|u_k(\cdot+z)-u_k\|_{L^p}^p\,\frac{\rho_{t_k}(z)}{|z|^p}
\di z
\\
&=
\int_{\R^N}
\|u(\cdot+z)-u\|_{L^p}^p\,\kappa(z)
\di z
=
[u]_{W^{\kappa,p}}^p.
\end{split}
\end{equation*}
In particular, if $\liminf\limits_{k\to\infty}\mathscr F_{t_k,p}(u_k)<\infty$, then $u\in W^{\kappa,p}(\R^N)$.
The proof is complete.
\end{proof}

For the proof of \cref{resi:diretto_comp} we need the following simple estimate exploiting~\eqref{eqi:diretto_eps-delta}.

\begin{lemma}
\label{res:comp_sub_eps-delta}
Let $p\in[1,\infty)$ and $(\rho_t)_{t\in I}\subset L^1_{\rm loc}(\R^N)$.
If~\eqref{eqi:diretto_eps-delta} holds,
then for every $\e>0$ there exists $\delta>0$ such that, letting $\eta_\delta=\chi_{B_\delta}/|B_\delta|$,
\begin{equation*}
\|\eta_\delta*u-u\|_{L^p}^p
\le 
\frac{\e}{|B_1|}\,\mathscr F_{t,p}(u)
\end{equation*}
for every $t\in(0,\delta)$ and $u\in L^p(\R^N)$.
\end{lemma}

\begin{proof}
Let $u\in L^p(\R^N)$ and $\e,\delta>0$ be as in~\eqref{eqi:diretto_eps-delta}. By Jensen's inequality, we have
\begin{equation*}
\begin{split}
\|\eta_\delta*u-u\|_{L^p}^p
\le 
\frac{1}{|B_\delta|}
\int_{B_\delta}
\|u(\cdot-z)-u\|_{L^p}^p\di z
=
\frac{1}{|B_1|}
\int_{B_\delta}
\frac{\|u(\cdot+z)-u\|_{L^p}^p}{\delta^N}\di z.
\end{split}
\end{equation*}
In virtue of~\eqref{eqi:diretto_eps-delta}, we thus infer that 
\begin{equation*}
\begin{split}
\|\eta_\delta*u-u\|_{L^p}^p
\le 
\frac{\e}{|B_1|}
\int_{B_\delta}
\|u(\cdot+z)-u\|_{L^p}^p
\,
\frac{\rho_t(z)}{|z|^p}\di z
=
\frac{\e}{|B_1|}
\,\mathscr F_{t,p}(u),
\end{split}
\end{equation*}
concluding the proof.
\end{proof}

\begin{proof}[Proof of \cref{resi:diretto_comp}]
By assumption, the sequence $(u_k)_{k\in\N}$ is bounded in $L^p(\R^N)$.
Thus, letting $\eta_\delta=\chi_{B_\delta}/{|B_\delta|}\in L^1(\R^N)$ for $\delta>0$, by~\cite{B11}*{Cor.~4.28} the sequence $(\eta_\delta*u_k)_{k\in\N}$ is locally precompact in $L^p(\R^N)$ for each $\delta>0$.
As a consequence, the sequence $(\eta_\delta*u_k)_{k\in\N}$ is totally bounded in $L^p(E)$ for every compact set $E\subset\R^N$.
By \cref{res:comp_sub_eps-delta}, also the sequence $(u_k)_{k\in\N}$ is totally bounded in $L^p(E)$ for every compact set $E\subset\R^N$, which implies that the sequence $(u_k)_{k\in\N}$ is is locally precompact in $L^p(\R^N)$.
Finally, if $u$ is an $L^p_{\rm loc}(\R^N)$ limit of $(u_k)_{k\in\N}$, then, up to subsequences, by Fatou's Lemma we have 
\begin{equation*}
\begin{split}
\sup_{k\in\N}\mathscr F_{t_k,p}(u_k)
&\ge 
\liminf_{k\to\infty}\mathscr F_{t_k,p}(u_k)
\ge 
\int_{\R^N}\liminf_{k\to\infty}
\left(
\|u_k(\cdot+z)-u_k\|_{L^p}^p\,\frac{\rho_{t_k}(z)}{|z|^p}
\right)
\di z
\\
&\ge
\int_{\R^N}
\|u(\cdot+z)-u\|_{L^p}^p\,\kappa(z)
\di z
=
[u]_{W^{\kappa,p}}^p,
\end{split}
\end{equation*}
showing that  $u\in W^{\kappa,p}(\R^N)$ and
concluding the proof.
\end{proof}

\section{Application to heat kernels}
\label{sec:heat}

In this section we apply the results of \cref{sec:special,sec:non-local} to families $(\rho_t)_{t\in I}$ induced by heat-type kernels, proving \cref{resi:heat,resi:frac_heat}.

\subsection{Heat kernel}
\label{subsec:heat}

We let $(\heat_t)_{t>0}\colon\R^N\to(0,\infty)$ be the \emph{heat kernel}, which is given by  
\begin{equation}
\label{eq:heat_kernel}
\heat_t(x)=\frac{e^{-\frac{|x|^2}{4t}}}{(4\pi t)^{\frac N2}},
\end{equation}
for all $x\in\R^N$ and $t>0$.
Given $p\in[1,\infty]$, we define the \emph{heat semigroup} 
\begin{equation*}
(\Heat_t)_{t>0}\colon L^p(\R^N)\to L^p(\R^N)
\end{equation*}
by letting
\begin{equation}
\label{eq:heat_group}
\Heat_tu=\heat_t*u
\end{equation}
for all $u\in L^p(\R^N)$ and $t>0$.
Note that~\eqref{eq:heat_group} makes sense by Young's inequality, since $\heat_t\in L^1(\R^N)$ for all $t>0$ due to~\eqref{eq:heat_kernel}.

We can now deal with the proof of \cref{resi:heat}.

\begin{proof}[Proof of \cref{resi:heat}]
Letting $K(x)=\heat_1(x)$ for all $x\in\R^N$ and $\beta(t)=t^{-\frac12}$ for all $t>0$, we get that $K_t(x)=\heat_t(x)$ for all $x\in\R^N$ and $t>0$.
By radial symmetry, we can compute 
\begin{equation*}
\||\cdot|^p\,K\|_{L^1}
=
\frac{2^{p-1}N}{\sqrt\pi}\,\frac{\Gamma\left(\frac{N+p}{2}\right)}{\Gamma\left(\frac{N+1}{2}\right)}.
\end{equation*}
Since $|\cdot|^p\,K\in L^1(\R^N)$ for every $p\in[1,\infty)$ due to~\eqref{eq:heat_kernel}, the conclusion hence follows from \cref{res:asym_Kp,resi:compactness}.
We omit the simple computations.
\end{proof}

\subsection{Fractional heat kernel}
\label{sec:frac_heat_kernel}

Given $s\in(0,1)$, we let $(\heat_t^s)_{t>0}\colon\R^n\to(0,\infty)$ be the \emph{fractional heat kernel}.
For $s=\frac12$, it is known that 
\begin{equation}
\label{eq:frac_heat_1/2}
\heat_t^{\frac12}(x)
=
\frac{\Gamma\left(\frac{N+1}2\right)}{\pi^{\frac{N+1}2}}
\,
\frac{t}{\left(t^2+|x|^2\right)^{\frac{N+1}{2}}}
\end{equation}
for all $x\in\R^N$ and $t>0$, see~\cite{V18}*{Eq.~(2.2)}.
For $s\in(0,1)$, $s\ne\frac12$, the heat kernel does not have an explicit formula. 
Anyway, it is a smooth, positive, radially symmetric probability function obeying the scaling law 
\begin{equation}
\label{eq:frac_heat_scaling}
\heat_t^s(x)=t^{-\frac N{2s}}\,\heat_1^s\left(t^{-\frac1{2s}}x\right)
\end{equation}
for all $x\in\R^N$ and $t>0$.
Moreover, by~\cite{BG60}*{Th.~2.1}, we have that 
\begin{equation*}
\lim_{x\to\infty}
|x|^{N+2s}\,\heat_1^s(x)
=
\zeta_{N,s},
\end{equation*}
with
\begin{equation}
\label{eq:frac_lap_const}
\zeta_{N,s}
=
\frac{s\,4^s}{\pi^{\frac N2}}
\,
\frac{\Gamma\left(\frac{N}{2}+s\right)}{\Gamma(1-s)}.
\end{equation}
Therefore, thanks to~\eqref{eq:frac_heat_scaling}, we also get that 
\begin{equation}
\label{eq:lim_frac_heat_zero}
\lim_{t\to0^+}
\frac{\heat_t^s(x)}{t}
=
\frac{\zeta_{N,s}}{|x|^{N+2s}}
\end{equation}
for all $x\in\R^N\setminus\set*{0}$,
where $\zeta_{N,s}$ is as in~\eqref{eq:frac_lap_const}, and that there exists $C_{N,s}>0$ such that 
\begin{equation}
\label{eq:frac_heat_bounds}
\frac{C_{N,s}^{-1}\,t}{\left(t^{\frac1s}+|x|^2\right)^{\frac{N+2s}{2}}}
\le 
\heat_t^s(x)
\le 
\frac{C_{N,s}\,t}{\left(t^{\frac1s}+|x|^2\right)^{\frac{N+2s}{2}}}
\end{equation}
for all $x\in\R^N$ and $t>0$.
The fractional heat kernel enjoys the following relation with the integer heat kernel~\eqref{eq:heat_kernel}.
Actually, such relation is a particular case of the general approach sketched in~\cite{BBKRSV80}*{Sec.~1.1.4}.
For each $s\in(0,1)$, there exists a family of probability densities $(\eta_t^{s})_{t>0}$ on $(0,\infty)$ such that
\begin{equation}
\label{eq:hhs}
\heat_t^s(x)
=
\int_0^\infty\heat_\tau(x)\,\eta_t^s(\tau)\di\tau
\quad
\text{for all}\ x\in\R^N\ \text{and}\ t>0.
\end{equation}
As proved in~\cite{AV14}*{Sec.~2}, the family $(\eta_t^{s})_{t>0}$ satisfies
\begin{equation}
\label{eq:hhs_int}
\int_0^\infty\tau^\alpha\,\eta_1^s(\tau)\di\tau
=
\frac{\Gamma\left(1-\frac\alpha s\right)}{\Gamma\left(1-\alpha\right)}
\end{equation}
for all $\alpha\in\left(-\infty,s\right)$.
 
As above, given $s\in(0,1)$ and $p\in[1,\infty]$, we define the \emph{fractional heat semigroup} 
\begin{equation*}
(\Heat_t^s)_{t>0}\colon L^p(\R^N)\to L^p(\R^N)
\end{equation*}
by letting
\begin{equation}
\label{eq:frac_heat_group}
\Heat_t^su=\heat_t^s*u
\end{equation}
for all $u\in L^p(\R^N)$ and $t>0$.
We observe that~\eqref{eq:frac_heat_group} makes sense again by Young's inequality, since $\heat_t^s\in L^1(\R^N)$ for all $t>0$ and $s\in(0,1)$, owing to the bounds in~\eqref{eq:frac_heat_bounds}.

We can now deal with the proof of \cref{resi:frac_heat}.

\begin{proof}[Proof of \cref{resi:frac_heat}]
Letting $K^s(x)=\heat_1^s(x)$ for all $x\in\R^N$ and $\beta_s(t)=t^{-\frac s2}$ for all $t>0$, we get that $K_t^s(x)=\heat_t^s(x)$ for all $x\in\R^N$ and $t>0$.
We distinguish three cases.

\vspace{1ex}

\textit{Case $2s>p$}.
We have that $|\cdot|^p\,K^s\in L^1(\R^N)$, so the conclusion follows by \cref{res:asym_Kp}.
We just need to observe that, by radial symmetry, we can compute
\begin{equation*}
\||\cdot|^p\,K\|_{L^1}
=
\frac{2^{p-1}}{\sqrt\pi}
\,
\frac{\Gamma\left(\frac{N+p}{2}\right)}{\Gamma\left(\frac{N+1}{2}\right)}
\,
\frac{\Gamma\left(1-\frac p{2s}\right)}{\Gamma\left(1-\frac p2\right)}
\end{equation*}
exactly as in~\cite{AV16}*{Lem.~4.1}, thanks to the relation~\eqref{eq:hhs} and the formula~\eqref{eq:hhs_int}.
We omit the simple computations.
For the compactness part, we can invoke \cref{resi:compactness}.

\vspace{1ex}

\textit{Case $2s=p$}.
We have that $|\cdot|^{2s}\,K^s\notin L^1(\R^N)$. 
However, we can compute
\begin{equation*}
m_s(R)
=
m_{K^s,2s}(R)
=
N\omega_N
\int_0^R r^{N+2s-1}\,\heat_1^s(r\mathrm e_1)\di r
\quad
\text{for all}\ R>0.
\end{equation*}
Since 
\begin{equation*}
\frac{\mathrm{d}}{\mathrm{d}R}
\int_0^R r^{N+2s-1}\,\heat_1^s(r\mathrm e_1)\di r
=
R^{N+2s-1}\,\heat_1^s(R\mathrm e_1)
\sim
\frac{c_{N,s}}{R}
\end{equation*}
as $R\to\infty$,
where $c_{N,s}>0$ is as in~\eqref{eq:frac_lap_const}, we must have that 
\begin{equation*}
m_s(R)
\sim
N\omega_N\,c_{N,s}\log R
\end{equation*}
as $R\to\infty$.
Hence, recalling the definition in~\eqref{eq:phi_Kp}, we deduce that 
\begin{equation*}
\phi_{K^s,2s,2s}(t)
=
t\,m_s\left(t^{-\frac s2}\right)
\sim
N\omega_N\,c_{N,s}\,\frac s2\,t|\log t|
\end{equation*}
as $t\to0^+$.
The conclusion hence follows by observing that the family $(\rho_t)_{t>0}$ given by
\begin{equation*}
\rho_t^s(x)
=
\frac{|x|^p\,\heat_t^s(x)}{t|\log t|},
\end{equation*}
for all $x\in\R^N$ and $t>0$,
satisfies the properties in \cref{resi:bbm}\ref{itemi:bbm_kernels} (with no need of passing to a subsequence).
For the compactness part, we can invoke \cref{res:ponce_log}.

\vspace{1ex}

\textit{Case $2s<p$}.
For the validity of the limit we can rely on \cref{resi:diretto_lim}, thanks to the limit in~\eqref{eq:lim_frac_heat_zero} and the bounds in~\eqref{eq:frac_heat_bounds}, while the compactness part follows from \cref{resi:diretto_comp}.
We omit the simple details.
\end{proof}

\begin{remark}[On the constants in \cref{resi:frac_heat} for $p=1$]
\label{rem:KL25_const}
We observe that, for $p=1$, the constants in \cref{resi:frac_heat} were computed in~\cite{KL25} with a completely different method, based on the approach of~\cite{ADM11}.
We observe that our method is more direct and flexible, yielding the values of the constants for all $p\in[1,\infty)$.
\end{remark}

%\begin{remark}[On the proof of~\cite{KL25}*{Th.~2.1} for $s\ne\frac12$]
%\label{rem:KL25_gap}
%The proof of~\cite{KL25}*{Th.~2.1} in the case $s\ne\frac12$, which is detailed in~\cite{KL25}*{Sec.~2.1}, seems incorrect. 
%Indeed, although the statement and the proof of~\cite{KL25}*{Lem.~2.4} are correct, this result cannot be used to ensure the validity of the assumptions in~\cite{KL25}*{Lem.~2.2}, as the family $(\chi_{E_i})_{i\in\N}$ \textit{a priori} depends on the chosen sequence $(t_i)_{i\in\N}$, being $E_i=E_{t_i}$ for $i\in\N$.  
%\end{remark}

\subsection{General heat-type kernels}
We conclude this section by generalizing~\cite{AV16}*{Th.~1.1}, see \cref{res:acuna} below. 
To state our result, we need to introduce some notation.

Let $\mathscr A$ be a set of indices and consider a family of non-negative continuous functions $(\heat_t^\alpha)_{t>0}\colon\R^N\to[0,\infty)$ for $\alpha\in\mathscr A$. Assume that $\heat_t^\alpha$ is radially symmetric with $\|\heat_t^\alpha\|_{L^1}=1$ for each $t>0$ and $\alpha\in\mathscr A$.
Moreover, assume that there is $\beta>0$ such that 
\begin{equation}
\label{eq:acuna_scaling}
\heat_t^\alpha(x)
=
t^{-\beta N}\,\heat_1^\alpha(t^{-\beta}x)
\end{equation}
for all $x\in\R^N$, $t>0$ and $\alpha\in\mathscr A$.
Given $p\in[1,\infty)$, we let 
\begin{equation*}
\mathscr A_{p}
=
\set*{\alpha\in\mathscr A : |\cdot|^p\,\heat_1^\alpha\in L^1(\R^N)}
\end{equation*}
Moreover, given $\zeta>0$ and $C\ge1$, we let $\mathscr B_p^{\zeta,C}\subset\mathscr A$ be the subset of indices $\alpha\in\mathscr A$ such that 
\begin{equation}
\label{eq:acuna_h_lim}
\lim_{|x|\to\infty}|x|^{N+p}\,\heat^\alpha_1(x)=\zeta
\end{equation}
and 
\begin{equation}
\label{eq:acuna_bounds}
\frac{C^{-1}}{(1+|x|)^{N+p}}
\le
\heat_1^\alpha(x)
\le 
\frac{C}{(1+|x|)^{N+p}}
\end{equation}
for all $x\in\R^N$.
Finally, we let $\mathscr C_p\subset\mathscr A$ be the subset of indices $\alpha\in\mathscr A$ such that there exist $\bar t_\alpha,c_\alpha>0$, a Borel function $\psi_\alpha\colon(0,\infty)\to(0,\infty)$ and a measurable function $\kappa_\alpha\colon\R^N\to[0,\infty)$ such that
\begin{equation}
\label{eq:acuna_k_frac_up}
\frac{\heat_t^\alpha(x)}{\psi_\alpha(t)}\le c_\alpha\kappa_\alpha(x)
\end{equation} 
for a.e.\ $x\in\R^N$ and $t\in(0,\bar t_\alpha)$
and 
\begin{equation}
\label{eq:acuna_lim_k}
\lim_{t\to0^+}
\frac{\heat_t^\alpha(x)}{\psi_\alpha(t)}
=
\kappa_\alpha(x)
\end{equation}
for a.e.\ $x\in\R^N$.
We also let $\widetilde{\mathscr C}_p\subset\mathscr C_p$ be be the subset of indices $\alpha\in\mathscr C_p$ such that there exist $\hat t_\alpha,\hat c_\alpha>0$ and a measurable function $\hat\kappa_\alpha\colon\R^N\to[0,\infty]$ such that
\begin{equation}
\label{eq:acuna_bs}
\hat\kappa_\alpha\notin L^1(\R^N)
\quad
\text{and}
\quad
\hat\kappa_\alpha\in L^1(\R^N\setminus B_R)
\end{equation}
for all $R>0$,
and
\begin{equation}
\label{eq:acuna_k_frac_down}
\frac{\heat_t^\alpha(x)}{\psi_\alpha(t)}
\ge 
\hat c_\alpha\hat\kappa_\alpha(x)
\end{equation}   
for a.e.\ $x\in\R^N$ and $t\in(0,\hat t_\alpha)$.
As above, for each $\alpha\in\mathscr A$ and $p\in[1,\infty]$, we define 
\begin{equation*}
(\Heat_t^\alpha)_{t>0}\colon L^p(\R^N)\to L^p(\R^N)
\end{equation*}
by letting 
\begin{equation}
\label{eq:heat_group_av}
\Heat_t^\alpha u=\heat_t^\alpha*u
\end{equation}
for all $u\in L^p(\R^N)$ and $t>0$.
We observe that~\eqref{eq:heat_group_av} makes sense by Young's inequality, since $\heat_t^\alpha\in L^1(\R^N)$ for all $t>0$ and $\alpha\in\mathscr A$ by assumption.

We are now ready to state our result, improving and generalizing~\cite{AV16}*{Th.~1.1}. 

\begin{theorem}
\label{res:acuna}
Let $p\in[1,\infty)$ and let $(\Heat_t^\alpha)_{t>0,\alpha\in\mathscr A}$ be as above. 
Let $\varsigma_\alpha\colon I\to[0,\infty)$ be defined as
\begin{equation*}
\varsigma_\alpha(t)
=
\begin{cases}
t^{\beta p} & \text{if}\ \alpha\in\mathscr A_p,
\\
t^{\beta p}|\log t| & \text{if}\ \alpha\in\mathscr B_p^{\zeta,C},
\\
\psi_\alpha(t) & \text{if}\ \alpha\in\widetilde{\mathscr C}_p,
\end{cases}
\end{equation*}
for all $\alpha\in\mathscr A$, with $\zeta>0$ and $C\ge1$.
The limits
\begin{equation*}
\lim_{t\to0^+}\int_{\R^N}
\frac{\Heat_t^\alpha(|u-u(x)|^p)(x)}{\varsigma_\alpha(t)}\di x
=
\begin{cases}
\displaystyle
\frac{c_{N,p}}{N\,\||\cdot|^p \,\heat_1^\alpha\|_{L^1}}
\,
\|Du\|_{L^p}^p
&
\text{in}\ {\sob{p}}(\R^N)\ \text{if}\ \alpha\in\mathscr A_p,
\\[5ex]
\zeta\,\beta\,\omega_{N}\,c_{N,p}
\,
\|Du\|_{L^p}^p
&
\text{in}\ {\sob{p}}(\R^N)\ \text{if}\ \alpha\in\mathscr B_p^{\zeta,C},
\\[5ex]
\displaystyle
[u]_{W^{\kappa_\alpha,p}}^p
&
\text{in}\ W^{\kappa_\alpha,p}(\R^N)\ \text{if}\ \alpha\in\mathscr C_p,
\end{cases}
\end{equation*}
where
\begin{equation*}
c_{N,p}
=
\frac{2\,\Gamma\left(\frac{p+1}{2}\right)\Gamma\left(\frac{N+1}{2}\right)}{\Gamma\left(\frac{N+p}2\right)},
\end{equation*}
hold in the pointwise sense and in the $\Gamma$-sense with respect to the $L^p$ topology, and all the functionals on the left-hand sides are coercive on the respective spaces.
Moreover, if $(t_k)_{k\in\N}\subset I$ is infinitesimal and $(u_k)_{k\in\N}\subset L^p(\R^N)$ is such that
\begin{equation*}
\liminf_{k\to\infty}
\frac{1}{\varsigma_\alpha(t_k)}
\int_{\R^N}\Heat_{t_k}^\alpha(|u_k-u_k(x)|^p)(x)\di x<\infty,
\end{equation*}
then $(u_k)_{k\in\N}$ is locally precompact in $L^p(\R^N)$ and any of its $L^p_{\rm loc}(\R^N)$ limits is in ${\sob{p}}(\R^N)$ if $\alpha\in\mathscr A_p\cup\mathscr B_p^{\zeta,C}$ and in $W^{\kappa_\alpha,p}(\R^N)$ if $\alpha\in\widetilde{\mathscr C_p}$.
\end{theorem}

\begin{proof}
The proof is quite similar to that of \cref{resi:frac_heat}, so we simply sketch it.
As before, letting $K(x)=\heat_1^\alpha(x)$ for all $x\in\R^N$ and $\beta_\alpha(t)=t^{-\beta}$ for all $t>0$, we get that $K_t^\alpha(x)=\heat^\alpha_t(x)$ for all $x\in\R^N$ and $t>0$.
The case $\alpha\in\mathscr A_p$ directly follows from \cref{res:asym_Kp,resi:compactness}.
For the case $\alpha\in\mathscr B_p^{\zeta,C}$, we observe that, owing to~\eqref{eq:acuna_h_lim},  
\begin{equation*}
\phi_{K^\alpha,p}(t)
\sim
N\omega_N\,\zeta\beta\,t^{\beta p}|\log t|
\quad
\text{as}\ t\to0^+. 
\end{equation*}
Therefore, it is enough to check that the family $(\rho_t)_{t>0}$ given by 
\begin{equation*}
\rho_t(x)
=
\frac{|x|^p\,\heat_t^\alpha(x)}{t^{\beta p}|\log t|},
\quad
\text{for all}\ x\in\R^N\ \text{and}\ t>0,
\end{equation*}
satisfies the properties in \cref{resi:bbm}\ref{itemi:bbm_kernels} (with no need of passing to a subsequence).
For the compactness part, thanks to~\eqref{eq:acuna_bounds}, it is enough to apply \cref{res:ponce_log}.
Finally, owing to~\eqref{eq:acuna_k_frac_up} and~\eqref{eq:acuna_lim_k}, the case $\alpha\in\mathscr C_p$ follows from \cref{resi:diretto_lim}, while, owing also to~\eqref{eq:acuna_bs} and~\eqref{eq:acuna_k_frac_down}, the compactness in the case $\alpha\in\widetilde{\mathscr C}_p$ follows from~\cite{BS24}*{Th.~2.11}. 
\end{proof}

\begin{remark}[Comparison with~\cite{AV16}*{Th.~1.1}]
\label{rem:comparison_acuna}
We observe that~\cite{AV16}*{Th.~1.1} deals with the case $p=1$ only, and only in the case of bounded sets of finite perimeter. 
In addition, \cite{AV16}*{Th.~1.1(ii)} is weaker than \cref{res:acuna}, as~\cite{AV16}*{Th.~1.1(ii)} provides an upper estimate on the $\limsup$ and only under the stronger assumption that 
\begin{equation}
\label{eq:acuna_weak}
\heat_1^\alpha(x)=\frac{c}{(1+|x|^n)^m},
\quad
\text{for all $x\in\R^N$,}
\end{equation}
for some $c,n,m>0$ such that $mn=N+1$. 
Our result instead relies on~\eqref{eq:acuna_h_lim} and~\eqref{eq:acuna_bounds} only, which, in turn, in the case~\eqref{eq:acuna_weak}, naturally impose that $mn=N+p$.
We further remark that our renormalization of $(\heat_t^\alpha)_{\alpha\in\mathscr A,t>0}$ differs from the one adopted in~\cite{AV16}, as we require that $\|\heat_1^\alpha\|_{L^1}=1$ for all $\alpha\in\mathscr A$.
However, due to the scaling assumption~\eqref{eq:acuna_scaling}, corresponding to~\cite{AV16}*{Eq.~(1.1)}, the two approaches are completely equivalent.    
\end{remark}

\section{Heat content in Hilbert spaces}

\label{sec:hilbert}

In this last section we briefly provide a different point of view on the sufficient condition achieved in \cref{resi:bbm} in the Hilbertian framework.

\label{subsec:hilbert}

\subsection{Heat content in Hilbert spaces}

We recall some standard notation (for a more detailed account, see, e.g., \cite{FOT11}).
We let $\mathcal H$ be a Hilbert space with scalar product $(\,\cdot\,,\,\cdot\,)_{\mathcal H}$ and we let $(\Heat_t)_{t\ge0}$ be a \emph{strongly continuous semigroup of symmetric operators} on $\mathcal H$.

\begin{definition}[Semigroup content]
\label{def:content_hilbert}
The \emph{semigroup content} of $u\in\mathcal H$ is the map $[0,\infty)\ni t\to\content_t(u)\in[0,\infty)$ defined as 
\begin{equation*}
\content_t(u)
=
(\Heat_t u,u)_{\mathcal H}
\end{equation*}
for all $t\ge0$.
In particular, $0\le\content_t(u)\le\content_0(u)= (u,u)_{\mathcal H}$ for $t\ge0$ and $u\in\mathcal H$.
\end{definition}

We let $\lap$ be the \emph{generator} of the semigroup $(\Heat_t)_{t\ge0}$, which is given by 
\begin{equation}
\label{eq:generator}
\lap u=\lim_{t\to0^+}\frac{\Heat_tu-u}{t}
\quad
\text{in}\ \mathcal H,
\end{equation}
with domain $\mathcal D(\lap)=\set*{u\in\mathcal H: \text{$\lap u$ exists as a strong limit}}$.
We recall that $\lap$ is a non-positive definite self-adjoint operator, see~\cite{FOT11}*{Lem.~1.3.1}. 
The notions of convergence in $\mathcal D(\lap)\subset\mathcal H$ in the pointwise and in the $\Gamma$-sense with respect to the topology in~$\mathcal H$ are the natural analogues of \cref{def:pointwise_conv,def:gamma_conv}. 
We omit the plain statements.

With this notation in force, we can prove \cref{resi:hilbert}. 

\begin{proof}[Proof of \cref{resi:hilbert}]
Let $E_\cdot=(E_\lambda)_{\lambda\ge0}$ be the \emph{spectral representation} of $-\lap$, so that 
\begin{equation*}
(-\lap u,v)_{\mathcal H}=\int_0^\infty\lambda\di(E_\lambda u,v)_{\mathcal H}
\end{equation*}
for $u\in\mathcal D(\lap)$ and $v\in\mathcal H$.
By~\cite{FOT11}*{Lem.~1.3.2}, we can write 
\begin{equation*}
(\Heat_t u,v)_{\mathcal H}
=
\int_0^\infty e^{-t\lambda}\di(E_\lambda u,v)_{\mathcal H},
\end{equation*}
for all $t\ge0$ and $u,v\in\mathcal H$,
and thus
\begin{equation*}
\content_0(u)-\content_t(u)
=
(u-\Heat_t u,u)_{\mathcal H}
=
\int_0^\infty(1-e^{-t\lambda})\di(E_\lambda u,u)
\end{equation*}
for all $t\ge0$ and $u\in\mathcal H$.
Since $1-s\le e^{-s}$ for all $s\ge0$, we can hence estimate
\begin{equation*}
\content_0(u)-\content_t(u)
\le 
t
\int_0^\infty \lambda\di(E_\lambda u,u)
=
t
(-\lap u,u)_{\mathcal H}
\end{equation*}
for all $t\ge0$ and $u\in\mathcal D(\lap)$, yielding~\ref{itemi:hilbert_limsup}. 

Now let $(u_k)_{k\in\N}\subset\mathcal H$, $u\in\mathcal H$ and $(t_k)_{k\in\N}\subset(0,\infty)$ be as in~\ref{itemi:hilbert_liminf}.
We define the measures $\mu_k,\mu\in\mathscr M_{\rm loc}^+((0,\infty))$ by letting $\mu_k=(E_{\cdot}u_k,u_k)_{\mathcal H}$ for all $k\in\N$ and $\mu=(E_{\cdot}u,u)_{\mathcal H}$.
Since $u_k\to u$ in~$\mathcal H$ as $k\to\infty$, we infer that $\mu_k\weakstarto\mu$ in $\mathscr M_{\rm loc}((0,\infty))$ as $k\to\infty$. 
Therefore, given $R>0$, we can estimate
\begin{equation*}
\begin{split}
\frac{\content_0(u_k)-\content_{t_k}(u_k)}{t_k}
\ge 
\int_0^R
\frac{1-e^{-t_k\lambda}}{t_k\lambda}\,\lambda\di\mu_k(\lambda)
\ge 
\frac{1-e^{-t_kR}}{t_kR}
\int_0^R
\lambda\di\mu_k(\lambda)
\end{split}
\end{equation*}
for all $k\in\N$. 
By Tonelli's Theorem and Fatou's Lemma, we have that\begin{equation*}
\liminf_{k\to\infty}
\int_0^R
\lambda\di\mu_k(\lambda)
=
\liminf_{k\to\infty}
\int_0^R\mu_k((t,R))\di t
\ge 
\int_0^R\mu((t,R))\di t
=
\int_0^R\lambda\di\mu(\lambda)
\end{equation*}
for all $R>0$. 
As a consequence, we get that
\begin{equation*}
\liminf_{k\to\infty}
\frac{\content_0(u_k)-\content_{t_k}(u_k)}{t_k}
\ge 
\int_0^R\lambda\di\mu(\lambda)
\end{equation*} 
for all $R>0$.
Letting $R\to\infty$ in the above inequality, we conclude that 
\begin{equation*}
\liminf_{k\to\infty}
\frac{\content_0(u_k)-\content_{t_k}(u_k)}{t_k}
\ge 
\int_0^\infty\lambda\di\mu(\lambda)
=
(-\lap u,u)_{\mathcal H},
\end{equation*}
proving~\ref{itemi:hilbert_liminf}. 
The rest of the statement follows by choosing $u_k=u$ for all $k\in\N$.
\end{proof}

\subsection{A quicker approach via Fourier transform}
\label{subsec:fourier}

We now provide an alternative quicker proof of \cref{resi:hilbert} for $\mathcal H=L^2(\R^N)$ via the Fourier transform.

We let 
\begin{equation*}
\mathcal F(u)(\xi)
=
\hat u(\xi)
=\int_{\R^N}e^{-2\pi ix\cdot\xi}\,u(x)\di x,
\end{equation*}
for all $\xi\in\R^N$,
be the \emph{Fourier transform} of $u\in L^1(\R^n)$.
As customary, we extend  the Fourier transform to a unitary transformation on $L^2(\R^N)$ keeping the same notation.

Given a measurable function $\sus\colon\R^N\to[0,\infty]$, the family
\begin{equation*}
(\Heat_t)_{t\ge0}\colon L^2(\R^N)\to L^2(\R^N),
\end{equation*} 
defined as 
\begin{equation}
\label{eq:suscalore}
(\Heat_t u,v)_{L^2}
=
\int_{\R^N}e^{-\sus(\xi)t}\,\hat u(\xi)\cdot\overline{\hat v(\xi)}\di\xi
\end{equation} 
for all $t\ge0$ and $u,v\in L^2(\R^N)$, yields a strongly continuous semigroup of symmetric operators on $L^2(\R^N)$.
As in \cref{def:content_hilbert}, we can thus define the semigroup content corresponding to~\eqref{eq:suscalore} as
\begin{equation*}
\content_t(u)
=
(\Heat_tu,u)_{L^2}
=
\int_{\R^N}e^{-\sus(\xi)t}\,|\hat u(\xi)|^2\di\xi
\end{equation*}
for all $t\ge0$ and $u\in L^2(\R^N)$. 
Moreover, as in~\eqref{eq:generator}, the  generator of~\eqref{eq:suscalore} is given by
\begin{equation*}
(-\lap u,v)_{L^2}
=
\int_{\R^N}\sus(\xi)\,\hat u(\xi)\cdot\overline{\hat v(\xi)}\di\xi,
\quad
v\in L^2(\R^N),
\end{equation*}   
for all $u\in\mathcal D(\lap)$, where $\mathcal D(\lap)=\set*{u\in L^2(\R^N): \sus |\hat u|\in L^2(\R^N)}$.
As usual, we can interpret~$\sus$ as the \emph{Fourier symbol} of the (non-negative) operator $-\lap$.

\begin{proof}[Alternative proof of \cref{resi:hilbert}]
Since we have 
\begin{equation*}
\content_0(u)-\content_t(u)
=
\int_{\R^N}(1-e^{-\sus(\xi)t})|\hat u(\xi)|^2\di\xi
\end{equation*}
for all $t\ge0$ and $u\in L^2(\R^N)$,
by the Dominated Convergence Theorem we infer that 
\begin{equation*}
\lim_{t\to0^+}
\frac{\content_0(u)-\content_t(u)}{t}
=
\int_{\R^N}\sus(\xi)\,|\hat u(\xi)|^2\di\xi
=
(-\lap u,u)_{L^2}
\end{equation*}
for all $u\in\mathcal D(\lap)$.
Now let $(t_k)_{k\in\N}\subset(0,\infty)$ be infinitesimal and $(u_k)_{k\in\N}\subset L^2(\R^N)$ be such that $u_k\to u$ in $L^2(\R^N)$ as $k\to\infty$ with $u\in\mathcal D(\lap)$.
As a consequence, also $\hat u_k\to\hat u$ in $L^2(\R^N)$ as $k\to\infty$ and thus, owing to Plancherel's Theorem, up to passing to a subsequence, $\hat u_k(\xi)\to\hat u(\xi)$ for a.e.\ $\xi\in\R^N$ as $k\to\infty$.
Then, by Fatou's Lemma, we get 
\begin{equation*}
\liminf_{k\to\infty}\frac{\content_0(u_k)-\content_{t_k}(u_k)}{t_k}
=
\liminf_{k\to \infty}
\int_{\R^N}\frac{1-e^{-\sus(\xi)t_k}}{t_k}\,|\hat u_k(\xi)|^2\di\xi
\ge 
\int_{\R^N}\sus(\xi)|\hat u(\xi)|^2\di\xi,
\end{equation*}   
readily yielding the  conclusion.
\end{proof}

\begin{remark}[Application to heat kernels]
\label{rem:fourier_heat}
The choice $\sus(\xi)=4\pi^2|\xi|^2$ for all $\xi\in\R^N$ yields that $\Heat_tu=\heat_t*u$ for all $t\ge0$ and $u\in L^2(\R^N)$ as in~\eqref{eq:heat_group}, and $\lap=\Delta$ is the Laplacian operator.
Similarly, given $s\in(0,1)$, the choice $\sus(\xi)=(2\pi|\xi|)^{2s}$ for all $\xi\in\R^N$, yields that $\Heat_t^su=\heat_t^s*u$ for all $t\ge0$ and $u\in L^2(\R^N)$ as in~\eqref{eq:frac_heat_group}, and $\lap=-(-\Delta)^{s}$ is the \emph{fractional Laplacian operator}, see~\cite{V18} for instance.
\end{remark}

\begin{remark}[Characteristic functions]
In the setting of \cref{subsec:fourier}, \cref{resi:hilbert} can be applied to characteristic functions of sets with finite volume.
In particular, in the case of heat kernels as in \cref{rem:fourier_heat}, this point of view allows to recover the results of~\cite{KL25} for $s\in\left(0,\frac12\right)$.
For finer results in this direction, we also refer to~\cites{BG24}. 
\end{remark}

\begin{remark}[Non-negativity assumption]
It is worth observing that the non-negativity assumption $-L\ge0$ in \cref{subsec:hilbert} (or, analogously, the fact that $\lambda\ge0$ in~\eqref{eq:suscalore} in \cref{subsec:fourier}) plays a crucial role in the proof of \cref{resi:hilbert}.
We do not know if such non-negativity assumption can be dropped or, at least, relaxed.
Similar considerations can be made concerning the non-negativity assumption $\rho_t\ge0$ for $t\in I$ in the case of the functionals $(\mathscr F_{t,p})_{t\in I}$ in~\eqref{eq:F_energy} for $p\in[1,\infty)$.
We refer to~\cites{AB98a,AB98b} for related discussions.
The authors thank Giovanni Alberti for several observations about these aspects.
\end{remark}

\subsection*{Data availability statement}
No new data were created or analysed in this study. Data sharing is not applicable to this article.

\subsection*{Conflict of interest}
All authors declare that they have no conflicts of interest.

%%% BIBLIO %%%

\end{document}